\documentclass[10pt,a4paper]{amsart}
\usepackage[latin9]{inputenc}
\usepackage{amsthm}
\usepackage{amstext}
\usepackage{amssymb}
\usepackage{esint}
\usepackage[unicode=true,pdfusetitle,
 bookmarks=true,bookmarksnumbered=false,bookmarksopen=false,
 breaklinks=false,pdfborder={0 0 1},backref=false,colorlinks=false]
 {hyperref}

\makeatletter
\numberwithin{equation}{section}
\numberwithin{figure}{section}
\theoremstyle{plain}
\newtheorem{thm}{\protect\theoremname}
  \theoremstyle{definition}
  \newtheorem{defn}[thm]{\protect\definitionname}
  \theoremstyle{remark}
  \newtheorem{rem}[thm]{\protect\remarkname}
  \theoremstyle{plain}
  \newtheorem{lem}[thm]{\protect\lemmaname}

\usepackage{amssymb,amsthm,amsmath,amsfonts,amscd}
\usepackage{amsaddr}
\usepackage{graphicx}
\usepackage{url}
\usepackage{color}
\usepackage[active]{srcltx}
\usepackage[matrix,arrow]{xy}
\usepackage{mathrsfs}
\usepackage{enumerate}
\usepackage{amsopn} 
\usepackage{bbm} 
\usepackage{cite}
\usepackage{hyperref}
\allowdisplaybreaks[1]
\usepackage[english]{babel}
\usepackage{bbm}
\usepackage{stmaryrd}
\usepackage{mdef}
\date{\today}

\thanks{{\bf Acknowledgements:}  We thank the DFG for support through the research project ``Random dynamical systems and regularization by noise for stochastic partial differential equations'' and through CRC 701.  B.G. thanks Jonas Tölle for stimulating discussion.}
\keywords{singular degenerate SPDE, multivalued SPDE, self-organized criticality, stochastic fast diffusion, sign fast diffusion, regularity, stochastic variational inequalities}
\subjclass[2010]{Primary: 60H15; Secondary: 76S05}

\makeatother

  \providecommand{\definitionname}{Definition}
  \providecommand{\lemmaname}{Lemma}
  \providecommand{\remarkname}{Remark}
\providecommand{\theoremname}{Theorem}

\begin{document}

\title[Singular-degenerate multivalued SFDE]{Singular-degenerate multivalued stochastic fast diffusion equations}

\author{Benjamin Gess}

\address{Department of Mathematics \\
University of Chicago \\
USA }

\email{gess@uchicago.edu}

\author{Michael Röckner}

\address{Faculty of Mathematics \\
University of Bielefeld \\
Germany}

\email{roeckner@mathematik.uni-bielefeld.de}
\begin{abstract}
We consider singular-degenerate, multivalued stochastic fast diffusion equations with multiplicative Lipschitz continuous noise. In particular, this includes the stochastic sign fast diffusion equation arising from the Bak-Tang-Wiesenfeld model for self-organized criticality. A well-posedness framework based on stochastic variational inequalities (SVI) is developed, characterizing solutions to the stochastic sign fast diffusion equation, previously obtained in a limiting sense only. Aside from generalizing the SVI approach to stochastic fast diffusion equations we develop a new proof of well-posedness, applicable to general diffusion coefficients. In case of linear multiplicative noise, we prove the existence of (generalized) strong solutions, which entails higher regularity properties of solutions than previously known.
\end{abstract}
\maketitle

\section{Introduction}

We consider singular-degenerate, multivalued stochastic fast diffusion equations (SFDE) of the type 
\begin{align}
dX_{t} & \in\D(|X_{t}|^{m-1}X_{t})dt+B(t,X_{t})dW_{t},\label{eq:SFDE}\\
X_{0} & =x_{0}\nonumber 
\end{align}
with $m\in[0,1]$, on bounded, smooth domains $\mcO\subseteq\R^{d}$ with zero Dirichlet boundary conditions, where $|r|^{-1}r:=\Sgn(r)$ denotes the maximal monotone extension of the sign function. In particular, we include the multivalued case $m=0$ and general diffusion coefficients $B$. In the following $W$ is a cylindrical Wiener process on some separable Hilbert space $U$ and the diffusion coefficients $B:[0,T]\times H^{-1}\times\O\to L_{2}(U,H^{-1})$ take values in the space of Hilbert-Schmidt operators $L_{2}(U,H^{-1})$, where $H^{-1}$ is the dual of $H_{0}^{1}(\mcO)$. 

Our main results are twofold: First, in the case of general diffusion coefficients and initial data we introduce a notion of stochastic variational inequalities (SVI) to \eqref{eq:SFDE} and establish a new method of proof of well-posedness. In particular, this new methods allows treatment of general diffusion coefficients, whereas previously the approach of SVI solutions was restricted to additive or linear multiplicative noise (cf.~\eqref{eq:TV_intro-additive}, \eqref{eq:TV_intro-lin-mult} below). In this sense, our results generalize those of \cite{BR13,BDPR09}. The second main result yields regularity properties of solutions in the case of linear multiplicative noise (cf.~\eqref{eq:SFDE-lin-mult} below). In particular, we prove the existence of strong solutions, which extends the results from \cite{G12} from the degenerate case $m>1$ to the singular case $m\in[0,1]$.

In the case $m>0$ a variational approach to \eqref{eq:SFDE} has been developed in \cite{RRW07} for $x_{0}\in L^{2}(\O;H^{-1})$ based on the coercivity property 
\[
_{V^{*}}\<\D(|v|^{m-1}v),v\>_{V}\ge c\|v\|_{V}^{m+1}\quad\forall v\in V=L^{m+1}(\mcO).
\]
In the multivalued limiting case $m=0$ two complications appear: First, the reflexivity of the energy space $L^{m+1}(\mcO)$ is lost, making the variational methods from \cite{RRW07} inapplicable in this case. Second, the operator $\D(|v|^{m-1}v)=\D\Sgn(v)$ becomes multivalued. Recently, for regular initial data $x_{0}\in L^{2}(\O;L^{2}(\mcO))$ an alternative variational approach to \eqref{eq:SFDE} has been developed in \cite{GT11}, proving well posedness for $m\in[0,1]$. However, for general initial conditions $x_{0}\in L^{2}(\O;H^{-1})$ solutions could be constructed in a limiting sense only. That is, it has been shown that for each approximating sequence $x_{0}^{n}\in L^{2}(\O;L^{2}(\mcO))$ with $x_{0}^{n}\to x$ in $L^{2}(\O;H^{-1})$ the corresponding variational solutions $X^{n}$ converge to a limit $X$ independent of the chosen approximating sequence $x_{0}^{n}$. The characterization of $X$ in terms of a generalized notion of solution to \eqref{eq:SFDE} remained open. This problem is solved in this paper by introducing a notion of stochastic variational inequalities for \eqref{eq:SFDE} and proving well-posedness in this framework. The limiting solution $X$ is thus characterized as an SVI solution to \eqref{eq:SFDE}.

The difficulties for \eqref{eq:SFDE} as described above are similar to the ones for the stochastic total variation flow
\begin{align}
dX_{t} & \in\div\left(\frac{\nabla X_{t}}{|\nabla X_{t}|}\right)dt+B(t,X_{t})dW_{t}\label{eq:TV_intro-2}\\
X_{0} & =x_{0}.\nonumber 
\end{align}
As for \eqref{eq:SFDE}, in case of regular initial data $x_{0}\in L^{2}(\O;H_{0}^{1}(\mcO))$ variational solutions have been constructed in \cite{GT11}. For general initial data $x_{0}\in L^{2}(\O;L^{2}(\mcO))$ solutions to \eqref{eq:TV_intro-2} could be constructed in a limiting sense only. In the special case of additive noise, i.e.
\begin{align}
dX_{t} & \in\div\left(\frac{\nabla X_{t}}{|\nabla X_{t}|}\right)dt+dW_{t}\label{eq:TV_intro-additive}
\end{align}
and assuming $d=1,2$, a notion of SVI solution to \eqref{eq:TV_intro-additive} has been introduced in \cite{BDPR09}. Only recently, well-posedness of SVI solutions and thus characterization of limiting solutions in the linear multiplicative case 

\begin{align}
dX_{t} & \in\div\left(\frac{\nabla X_{t}}{|\nabla X_{t}|}\right)dt+\sum_{k=1}^{\infty}f_{k}X_{t}d\b_{t}^{k}\label{eq:TV_intro-lin-mult}
\end{align}
has been shown in \cite{BR13}. In this sense, our results on well-posedness of SVI solutions to \eqref{eq:SFDE} parallel those of \cite{BR13,BDPR09} in the case of stochastic fast diffusion equations. In both cases \eqref{eq:TV_intro-additive} and \eqref{eq:TV_intro-lin-mult}, the SPDE may be transformed into a random PDE, i.e. a PDE with random coefficients. This technique is a crucial ingredient in the proofs given in \cite{BR13,BDPR09} and requires the restriction to either additive or linear multiplicative noise. In contrast to this, in the first part of this paper (see Section \ref{sec:SVI}) we consider \eqref{eq:SFDE} for general multiplicative noise and introduce an alternative method to prove well-posedness of SVI solutions that does not rely on a transformation into a random PDE. This allows to treat general noise, while significantly simplifying the proof as compared to \cite{BR13}. Moreover, in contrast to \cite{BDPR09} no restrictions on the dimension $d$ will be required.

In the second part of this paper (see Section \ref{sec:Strong-solutions}) we prove regularity properties for solutions to \eqref{eq:SFDE} in the case of linear multiplicative noise, i.e. for
\begin{align}
dX_{t} & \in\D(|X_{t}|^{m-1}X_{t})dt+\sum_{k=1}^{\infty}f_{k}X_{t}d\b_{t}^{k},\label{eq:SFDE-lin-mult}\\
X_{0} & =x_{0},\nonumber 
\end{align}
with $m\in[0,1]$, $d\in\N$ and $x_{0}\in L^{2}(\O;H^{-1})$. More precisely, we prove the existence of (generalized) strong solutions (cf. Definition \ref{def:strong_soln-general} below), in particular implying that $X$ takes values in the domain of $\D(|\cdot|^{m-1}\cdot)$, $dt\otimes\P$-almost everywhere. This extends regularity results obtained in \cite{G12} where the degenerate case $m\ge1$ was considered by entirely different methods. The case of singular diffusions ($m\in[0,1)$) could not be handled in \cite{G12} due to the singularity of the non-linearity $\phi(r)=|r|^{m-1}r$ at zero. Roughly speaking, this singularity has to be compensated by sufficient decay of the diffusion coefficients at zero; a problem not appearing in degenerate, non-singular cases treated in \cite{GR14}. This requires a careful choice of approximating problems and leads to entirely different methods than those developed in \cite{G12}. In particular, it turns out that different approximations of the nonlinearity $\phi$ need to be considered in the proof of regularity for \eqref{eq:SFDE-lin-mult} and in the proof of well-posedness of SVI solutions for \eqref{eq:SFDE}. We underline that also the methods developed in the second part of this paper do not rely on a transformation of \eqref{eq:SFDE-lin-mult} into a random PDE and therefore depend only loosely on the linear structure of the noise in \eqref{eq:SFDE-lin-mult}. In fact, as pointed out above, the crucial structural condition in \eqref{eq:SFDE-lin-mult} is \textit{not} the linearity of the diffusion coefficients, but their decay behavior at zero. 

Stochastic fast diffusion equations of the type \eqref{eq:SFDE} have been intensively investigated in recent years. For the single-valued case $m>0$ we refer to \cite{RRW07,L10} and the references therein. As the multivalued, limiting case $m=0$ is concerned, mostly the case of linear multiplicative noise \eqref{eq:SFDE-lin-mult} has been considered in the literature. Well-posedness for regular initial data $x_{0}\in L^{4}(\mcO)$ and $d=1,2,3$ was first proven in \cite{BDPR09-2}. Finite time extinction for \eqref{eq:SFDE-lin-mult} has been investigated in \cite{BDPR09-2,BDPR09-3,BDPR12,BR11,G13-5,RW13}. For bounded initial data $x_{0}\in L^{\infty}(\mcO)$ and finite driving noise, that is $f_{k}\equiv0$ for all $k$ large enough, the existence of strong solutions (cf. Definition \ref{def:strong_soln-general} below) to \eqref{eq:SFDE-lin-mult} has been proven in \cite{G13-5} by entirely different methods, relying on a transformation of \eqref{eq:SFDE-lin-mult} into a random PDE. Well-posedness for \eqref{eq:SFDE} with $m=1$ and with general multiplicative noise has been obtained in \cite{GT11} for the first time, proving well-posedness in terms of variational solutions for regular initial data $x_{0}\in L^{2}(\mcO)$. For background on the deterministic fast diffusion equation we refer to \cite{V06,V07} and the references therein.

\subsection{Notation}

In the following let $\mcO\subseteq\R^{d}$ be a bounded set with smooth boundary. $L^{p}:=L^{p}(\mcO)$ denotes the usual Lebesgue space with norm $\|\cdot\|_{L^{p}}$ and inner product $(\cdot,\cdot)_{2}$ if $p=2$. Further, $H_{0}^{1}:=H_{0}^{1}(\mcO)$ denotes the Sobolev space of order one in $L^{2}$ equipped with the inner product $(v,w)_{H_{0}^{1}}=(\nabla v,\nabla w)_{2}$ and norm $\|\cdot\|_{H_{0}^{1}}$ . Let $(H^{-1},(\cdot,\cdot)_{H^{-1}})$ with norm $\|\cdot\|_{H^{-1}}$ be the dual of $H_{0}^{1}$. Moreover, we let $C_{0}(\bar{\mcO})$ denote the set of all continuous functions on $\mcO$ vanishing at the boundary. In the proofs, as usual, constants may change from line to line.

\section{Stochastic variational inequalities\label{sec:SVI} }

In this section we consider stochastic singular fast diffusion equations of the type 
\begin{align}
dX_{t} & \in\D\left(|X_{t}|^{m-1}X_{t}\right)dt+B(t,X_{t})dW_{t},\label{eq:SVI-SFDE}\\
X_{0} & =x_{0}\nonumber 
\end{align}
for $m\in[0,1]$, on a bounded, smooth domain $\mcO\subseteq\R^{d}$ with Dirichlet boundary conditions and general diffusion coefficients $B$, in particular including additive and linear multiplicative noise. The precise definition of the nonlinear part on the right hand side of \eqref{eq:SVI-SFDE} including its domain, as well as the definition of a solution to \eqref{eq:SVI-SFDE} will be given below. We emphasize that the multivalued, limiting case $m=0$ is included. 

Here $W$ is a cylindrical Wiener process in some separable Hilbert space $U$ defined on a probability space $(\Omega,\mcF,\P)$ with normal filtration $(\mcF_{t})_{t\ge0}$ and the diffusion coefficients $B$ take values in the space of Hilbert-Schmidt operators $L_{2}(U,H)$. As compared to the regularity results obtained in Section \ref{sec:Strong-solutions} below, for this general choice of diffusion coefficients, we cannot expect (generalized) strong solutions to exist for arbitrary initial conditions $x_{0}\in H^{-1}$. Instead we introduce a notion of stochastic variational inequalities for \eqref{eq:SVI-SFDE} which we prove to uniquely characterize solutions.

We suppose that $B:[0,T]\times H^{-1}\times\O$ is progressively measurable and satisfies
\begin{align}
\|B(t,v)-B(t,w)\|_{L_{2}(U,H^{-1})}^{2} & \le C\|v-w\|_{H^{-1}}^{2}\quad\forall v,w\in H^{-1}\label{eq:lipschitz_noise}\\
\|B(t,v)\|_{L_{2}(U,L^{2})}^{2} & \le C(1+\|v\|_{L^{2}}^{2})\quad\forall v\in L^{2},\nonumber 
\end{align}
for some constant $C>0$ and all $(t,\o)\in[0,T]\times\O$.

Let $\mcM$ be the space of all signed Radon measures on $\mcO$ with finite total variation. For $\mu\in\mcM$ we let $|\mu|$ be its variation with total variation 
\[
\TV(\mu;\mcO)=|\mu|(\mcO).
\]
We define (cf. Appendix \ref{sec:heaviside_relaxation}) 
\begin{align*}
L^{m+1}\cap H^{-1} & :=\left\{ v\in L^{m+1}|\int vhdx\le C\|h\|_{H_{0}^{1}},\ \forall h\in C_{c}^{1}(\mcO)\ \text{for some }C\ge0\right\} \\
\mcM\cap H^{-1} & :=\left\{ \mu\in\mcM|\int h(x)d\mu(x)\le C\|h\|_{H_{0}^{1}},\ \forall h\in C_{c}^{1}(\mcO)\ \text{for some }C\ge0\right\} .
\end{align*}
Note that $\mcM\cap H^{-1}$ is known as the space of finite measures of bounded energy (cf. e.g. \cite{M85}). Clearly, we have $L^{m+1}\cap H^{-1}\subseteq H^{-1}$ and $\mcM\cap H^{-1}\subseteq H^{-1}$. For $m>0$ and $v\in H^{-1}$ we define $\vp:H^{-1}\to[0,\infty]$ by 
\[
\vp(v):=\begin{cases}
\frac{1}{m+1}\|v\|_{m+1}^{m+1} & ,\quad v\in L^{m+1}\cap H^{-1}\\
+\infty & ,\quad otherwise.
\end{cases}
\]
By Lemma \ref{lem:lsc_hull_lm} in Appendix \ref{sec:heaviside_relaxation} below, $\vp$ defines a convex, lower-semicontinuous function on $H^{-1}$. Moreover, by Lemma \ref{lem:lsc_hull_l1} 
\[
\vp(\mu):=\begin{cases}
\TV(\mu;\mcO) & ,\quad\mu\in\mcM\cap H^{-1}\\
+\infty & ,\quad otherwise.
\end{cases}
\]
is the lower-semicontinuous hull on $H^{-1}$ of $\|\cdot\|_{1}$ defined on $L^{1}\cap H^{-1}$ . For $m\ge0$ we set 
\[
\psi(r)=\frac{1}{m+1}|r|^{m+1},\quad r\in\R,
\]
and note that 
\begin{align*}
\partial\psi(r) & =|r|^{m}\Sgn(r)=:\phi(r),\quad r\in\R
\end{align*}
where $\Sgn$ is the maximal monotone, multivalued extension of the sign function.

Concerning the subgradient $\partial\vp$ of $\vp$ we have (see Lemma \ref{lem:subgradient} below)
\[
\partial\vp(v)\supseteq\{-\D w|w\in H_{0}^{1},w\in\phi(v)\text{ a.e.}\},
\]
for all $v\in L^{m+1}\cap H^{-1}$. Hence, we may rewrite \eqref{eq:SVI-SFDE} in the \textit{relaxed} form
\begin{align}
dX_{t} & \in-\partial\vp(X_{t})dt+B(t,X_{t})dW_{t},\label{eq:relaxed_form}\\
X_{0} & =x_{0}\nonumber 
\end{align}
and define (generalized) strong solutions according to Definition \ref{def:strong_soln-general} in Appendix \ref{sec:strong_solutions} below.
\begin{defn}
\label{def:SFDE_gen}Let $x_{0}\in L^{2}(\Omega;H^{-1}).$ An $\mcF_{t}$-adapted process $X\in L^{2}(\Omega;C([0,T];H^{-1}))$ is said to be an SVI solution to \eqref{eq:SVI-SFDE} if 
\begin{enumerate}
\item {[}Regularity{]} 
\begin{align*}
\vp(X) & \in L^{1}([0,T]\times\Omega).
\end{align*}

\item {[}Variational inequality{]} For each $\mcF_{t}$-progressively measurable process $G\in L^{2}([0,T]\times\O;H^{-1})$ and each $\mcF_{t}$-adapted process $Z\in L^{2}(\Omega;C([0,T];H^{-1}))\cap L^{2}([0,T]\times\Omega;L^{2})$ solving the equation
\[
Z_{t}-Z_{0}=\int_{0}^{t}G_{s}ds+\int_{0}^{t}B(s,Z_{s})dW_{s},\quad\forall t\in[0,T],
\]
we have
\begin{align}
 & \E\|X_{t}-Z_{t}\|_{H^{-1}}^{2}+2\E\int_{0}^{t}\vp(X_{r})dr\nonumber \\
 & \le\E\|x_{0}-Z_{0}\|_{H^{-1}}^{2}+2\E\int_{0}^{t}\vp(Z_{r})dr\label{eq:SFDE_gen_soln}\\
 & -2\E\int_{0}^{t}(G_{r},X_{r}-Z_{r})_{H^{-1}}dr+C\E\int_{0}^{t}\|X_{r}-Z_{r}\|_{H^{-1}}^{2}dr\quad\forall t\ge0,\nonumber 
\end{align}
for some $C>0$.
\end{enumerate}
\end{defn}
\begin{rem}
\label{rmk:varn_sol-1}If $(X,\eta)$ is a strong solution to \eqref{eq:SVI-SFDE} satisfying $\vp(X)\in L^{1}([0,T]\times\Omega)$ then $X$ is an SVI solution to \eqref{eq:SVI-SFDE}.\end{rem}
\begin{proof}
Definition \ref{def:SFDE_gen} (i) is satisfied by assumption. For (ii): Let $Z\in L^{2}(\Omega;C([0,T];H^{-1}))\cap L^{2}([0,T]\times\Omega;L^{2})$ be a solution to 
\begin{align*}
dZ_{t} & =G_{t}dt+B(t,Z_{t})dW_{t}
\end{align*}
for some $\mcF_{t}$-progressively measurable $G\in L^{2}([0,T]\times\O;H^{-1})$. Then Itô's formula implies:
\begin{align*}
\E\|X_{t}-Z_{t}\|_{H^{-1}}^{2}= & \E\|x_{0}-Z_{0}\|_{H^{-1}}^{2}+2\E\int_{0}^{t}(\eta_{r}-G_{r},X_{r}-Z_{r})_{H^{-1}}dr\\
 & +\E\int_{0}^{t}\|B(r,X_{r})-B(r,Z_{r})\|_{L_{2}(U,H^{-1})}^{2}dr\quad\forall t\in[0,T].
\end{align*}
Since $\eta_{r}\in-\partial\vp(X_{r})$ we have
\[
(\eta_{r},X_{r}-Z_{r})_{H^{-1}}\le\vp(Z_{r})-\vp(X_{r}),\quad dt\otimes d\P-\text{a.e.}
\]
which, using \eqref{eq:lipschitz_noise}, implies \eqref{eq:SFDE_gen_soln}.
\end{proof}
The main result of the current section is the proof of well-posedness of \eqref{eq:SVI-SFDE} in the sense of Definition \ref{def:SFDE_gen}:
\begin{thm}
\label{thm:periodic_main-1}Let $x_{0}\in L^{2}(\Omega;H^{-1}).$ Then there is a unique SVI solution $X$ to \eqref{eq:SVI-SFDE} in the sense of Definition \ref{def:SFDE_gen}. For two SVI solutions $X$, $Y$ with initial conditions $x_{0},y_{0}\in L^{2}(\Omega;H^{-1})$ we have
\[
\sup_{t\in[0,T]}\E\|X_{t}-Y_{t}\|_{H^{-1}}^{2}\le C\E\|x_{0}-y_{0}\|_{H^{-1}}^{2}.
\]
The unique SVI solution $X$ coincides with the limiting solution to \eqref{eq:SVI-SFDE} constructed in \cite{GT11}.\end{thm}
\begin{proof}
We construct SVI solutions to \eqref{eq:SVI-SFDE} by considering appropriate approximations by strong solutions. The specific form of the construction will also be a crucial ingredient in the proof of uniqueness. 

\textbf{Step 1:} Existence

We consider approximating SPDE of the form
\begin{align}
dX_{t}^{\ve} & =\ve\D X_{t}^{\ve}dt+\D\phi^{\ve}(X_{t}^{\ve})dt+B(t,X_{t}^{\ve})dW_{t},\label{eq:approx_SVI_constr}\\
X_{0}^{\ve} & =x_{0},\nonumber 
\end{align}
with $\ve>0$, $x_{0}\in L^{2}(\O;L^{2})$ and $\psi^{\ve},\phi^{\ve}$ as in Appendix \ref{sec:Moreau-Yosida}. By Lemma \ref{lem:fde-l2-test} there is a unique strong solution $X^{\ve}$ to \eqref{eq:approx_SVI_constr} satisfying 
\begin{equation}
\E\sup_{t\in[0,T]}\|X_{t}^{\ve}\|_{2}^{2}+2\ve\E\int_{0}^{T}\|X_{r}^{\ve}\|_{H_{0}^{1}}^{2}dr\le C(\E\|x_{0}\|_{2}^{2}+1),\label{eq:SFDE_strong_bound}
\end{equation}
for some $C>0$ independent of $\ve>0$. For two solutions $X^{\ve_{1}},X^{\ve_{2}}$ to \eqref{eq:approx_SVI_constr} with initial conditions $x_{0}^{1},x_{0}^{2}\in L^{2}(\O;L^{2})$ we have
\begin{align*}
e^{-Kt}\|X_{t}^{\ve_{1}}-X_{t}^{\ve_{2}}\|_{H^{-1}}^{2}= & \|x_{0}^{1}-x_{0}^{2}\|_{H^{-1}}^{2}\\
 & +2\int_{0}^{t}e^{-Kr}(\ve_{1}\D X_{r}^{\ve_{1}}-\ve_{2}\D X_{r}^{\ve_{2}},X_{r}^{\ve_{1}}-X_{r}^{\ve_{2}})_{H^{-1}}dr\\
 & +2\int_{0}^{t}e^{-Kr}(\D\phi^{\ve_{1}}(X_{r}^{\ve_{1}})-\D\phi^{\ve_{2}}(X_{r}^{\ve_{2}}),X_{r}^{\ve_{1}}-X_{r}^{\ve_{2}})_{H^{-1}}dr\\
 & +2\int_{0}^{t}e^{-Kr}(X_{r}^{\ve_{1}}-X_{r}^{\ve_{2}},B(r,X_{r}^{\ve_{1}})-B(r,X_{r}^{\ve_{2}}))_{H^{-1}}dW_{r}\\
 & +\int_{0}^{t}e^{-Kr}\|B(r,X_{r}^{\ve_{1}})-B(r,X_{r}^{\ve_{1}})\|_{L_{2}}^{2}dr\\
 & -K\int_{0}^{t}e^{-Kr}\|X_{r}^{\ve_{1}}-X_{r}^{\ve_{1}}\|_{H^{-1}}^{2}dr.
\end{align*}
Using \eqref{eq:monotone_Y_bound} we note that
\begin{align*}
(\D\phi^{\ve_{1}}(X_{r}^{\ve_{1}})-\D\phi^{\ve_{2}}(X_{r}^{\ve_{2}}),X_{r}^{\ve_{1}}-X_{r}^{\ve_{2}})_{H^{-1}} & =-\int_{\mcO}(\phi^{\ve_{1}}(X_{r}^{\ve_{1}})-\phi^{\ve_{2}}(X_{r}^{\ve_{2}}))(X_{r}^{\ve_{1}}-X_{r}^{\ve_{2}})dx\\
 & \le C(\ve_{1}+\ve_{2})(1+\|X_{r}^{\ve_{1}}\|_{2}^{2}+\|X_{r}^{\ve_{2}}\|_{2}^{2})
\end{align*}
and
\begin{align*}
(\ve_{1}\D X_{r}^{\ve_{1}}-\ve_{2}\D X_{r}^{\ve_{2}},X_{r}^{\ve_{1}}-X_{r}^{\ve_{2}})_{H^{-1}} & =\int_{\mcO}(\ve_{1}X_{r}^{\ve_{1}}-\ve_{2}X_{r}^{\ve_{2}})(X_{r}^{\ve_{1}}-X_{r}^{\ve_{2}})dx\\
 & \le C(\ve_{1}+\ve_{2})(\|X_{r}^{\ve_{1}}\|_{2}^{2}+\|X_{r}^{\ve_{2}}\|_{2}^{2}),
\end{align*}
$dt\otimes d\P$-a.e.. Thus,
\begin{align*}
e^{-Kt}\|X_{t}^{\ve_{1}}-X_{t}^{\ve_{2}}\|_{H^{-1}}^{2}\le & \|x_{0}^{1}-x_{0}^{2}\|_{H^{-1}}^{2}\\
 & +C(\ve_{1}+\ve_{2})\int_{0}^{t}(\|X_{r}^{\ve_{1}}\|_{2}^{2}+\|X_{r}^{\ve_{1}}\|_{2}^{2}+1)dr\\
 & +2\int_{0}^{t}e^{-Kr}(X_{r}^{\ve_{1}}-X_{r}^{\ve_{2}},B(r,X_{r}^{\ve_{1}})-B(r,X_{r}^{\ve_{2}}))_{H^{-1}}dW_{r}\\
 & +C\int_{0}^{t}e^{-Kr}\|X_{r}^{\ve_{1}}-X_{r}^{\ve_{1}}\|_{H^{-1}}^{2}dr\\
 & -K\int_{0}^{t}e^{-Kr}\|X_{r}^{\ve_{1}}-X_{r}^{\ve_{1}}\|_{H^{-1}}^{2}dr.
\end{align*}
Using the Burkholder-Davis-Gundy inequality and \eqref{eq:SFDE_strong_bound} we obtain
\begin{align}
\E\sup_{t\in[0,T]}e^{-Kt}\|X_{t}^{\ve_{1}}-X_{t}^{\ve_{2}}\|_{H^{-1}}^{2}\le & 2E\|x_{0}^{1}-x_{0}^{2}\|_{H^{-1}}^{2}\label{eq:stability_ic_SFDE}\\
 & +C(\ve_{1}+\ve_{2})(\E\|x_{0}^{1}\|_{2}^{2}+\E\|x_{0}^{2}\|_{2}^{2}+1),\nonumber 
\end{align}
for $K>0$ large enough. 

Let now $X^{\ve_{1}},X^{\ve_{2}}$ be strong solutions to \eqref{eq:approx_SVI_constr} with the same initial condition $x_{0}\in L^{2}(\O;L^{2})$. Then \eqref{eq:stability_ic_SFDE} implies 
\begin{align*}
\E\sup_{t\in[0,T]}e^{-Kt}\|X_{t}^{\ve_{1}}-X_{t}^{\ve_{2}}\|_{H^{-1}}^{2} & \le C(\ve_{1}+\ve_{2})(\E\|x_{0}\|_{2}^{2}+1)
\end{align*}
and thus
\[
\E\sup_{t\in[0,T]}\|X_{t}^{\ve_{}}-X_{t}\|_{H^{-1}}^{2}\to0\quad\text{for }\ve\to0
\]
for some $\mcF_{t}$-adapted process $X\in L^{2}(\O;C([0,T];H^{-1}))$ with $X_{0}=x_{0}$. For $x_{0}^{1},x_{0}^{2}\in L^{2}(\O;L^{2})$ taking the limit $\ve\to0$ in \eqref{eq:stability_ic_SFDE} yields
\begin{align}
\E\sup_{t\in[0,T]}e^{-Kt}\|X_{t}^{1}-X_{t}^{2}\|_{H^{-1}}^{2} & \le2E\|x_{0}^{1}-x_{0}^{2}\|_{H^{-1}}^{2}.\label{eq:stability_ic_SFDE-1}
\end{align}
Let now $X^{\ve,n}$ be the unique strong solution (cf. Lemma \ref{lem:fde-l2-test}) to 
\begin{align}
dX_{t}^{\ve,n} & =\ve\D X_{t}^{\ve,n}dt+\D\phi^{\ve}(X_{t}^{\ve,n})dt+B(t,X_{t}^{\ve,n})dW_{t},\label{eq:approx_SVI_constr-1}\\
X_{0}^{\ve,n} & =x_{0}^{n},\nonumber 
\end{align}
for some sequence $x_{0}^{n}\to x_{0}$ in $L^{2}(\O;H^{-1})$ with $x_{0}^{n}\in L^{2}(\O;L^{2})$. Using $ $\eqref{eq:stability_ic_SFDE} and \eqref{eq:stability_ic_SFDE-1} we obtain the existence of a sequence of $\mcF_{t}$-adapted processes $X^{n}\in L^{2}(\O;C([0,T];H^{-1}))$ with $X_{0}^{n}=x_{0}^{n}$ and an $\mcF_{t}$-adapted process $X\in L^{2}(\O;C([0,T];H^{-1}))$ with $X_{0}=x_{0}$ such that
\[
\E\sup_{t\in[0,T]}\|X_{t}^{\ve,n}-X_{t}^{n}\|_{H^{-1}}^{2}\to0\quad\text{for }\ve\to0
\]
and 
\[
\E\sup_{t\in[0,T]}\|X_{t}^{n}-X_{t}\|_{H^{-1}}^{2}\to0\quad\text{for }n\to\infty.
\]
Let now $G,Z$ be as in Definition \ref{def:SFDE_gen}. Itô's formula implies
\begin{align*}
\E\|X_{t}^{\ve,n}-Z_{t}\|_{H^{-1}}^{2}= & \E\|x_{0}^{n}-Z_{0}\|_{H^{-1}}^{2}\\
 & +2\E\int_{0}^{t}(\ve\D X_{t}^{\ve,n}+\D\phi^{\ve}(X_{t}^{\ve,n})-G_{r},X_{r}^{\ve,n}-Z_{r})_{H^{-1}}dr\\
 & +\E\int_{0}^{t}\|B(r,X_{r}^{\ve,n})-B(r,Z_{r})\|_{L_{2}}^{2}dr.
\end{align*}
For $v\in H^{-1}$ we set
\[
\vp^{\ve}(v)=\begin{cases}
\int_{\mcO}\psi^{\ve}(v)dx & ,\ v\in L^{2}\\
+\infty & ,\ \text{otherwise.}
\end{cases}
\]
Using convexity of $\psi^{\ve}$ we have
\[
(\D\phi^{\ve}(X_{r}^{\ve,n}),X_{r}^{\ve,n}-Z_{r})_{H^{-1}}+\vp^{\ve}(X_{r}^{\ve,n})\le\vp^{\ve}(Z_{r})
\]
$dt\otimes\P$-a.e.. Moreover,
\begin{align*}
(\ve\D X_{r}^{\ve,n},X_{r}^{\ve,n}-Z_{r})_{H^{-1}} & \le\ve\|\D X_{r}^{\ve,n}\|_{H^{-1}}\|X_{r}^{\ve,n}-Z_{r}\|_{H^{-1}}\\
 & \le\ve^{\frac{4}{3}}\|\D X_{r}^{\ve,n}\|_{H^{-1}}^{2}+\ve^{\frac{2}{3}}\|X_{r}^{\ve,n}-Z_{r}\|_{H^{-1}}^{2}
\end{align*}
$dt\otimes\P$-a.e.. Hence,
\begin{align}
 & \E\|X_{t}^{\ve,n}-Z_{t}\|_{H^{-1}}^{2}+2\E\int_{0}^{t}\vp^{\ve}(X_{r}^{\ve,n})dr\nonumber \\
\le & \E\|x_{0}^{n}-Z_{0}\|_{H^{-1}}^{2}+2\E\int_{0}^{t}\vp^{\ve}(Z_{r})dr\label{eq:svi-2}\\
 & -2\E\int_{0}^{t}(G_{r},X_{r}^{\ve,n}-Z_{r})_{H^{-1}}dr+C\E\int_{0}^{t}\|X_{r}^{\ve,n}-Z_{r}\|_{H^{-1}}^{2}dr\nonumber \\
 & +2\E\int_{0}^{t}\ve^{\frac{4}{3}}\|\D X_{r}^{\ve,n}\|_{H^{-1}}^{2}+\ve^{\frac{2}{3}}\|X_{r}^{\ve,n}-Z_{r}\|_{H^{-1}}^{2}dr.\nonumber 
\end{align}
Due to the definition of $\vp^{\ve}$ and \eqref{eq:yosida_convergence-2} we have
\begin{align}
|\vp^{\ve}(v)-\vp(v)| & \le C\ve(1+\vp(v))\label{eq:SVI_phi_error}\\
 & \le C\ve(1+\|v\|_{2}^{2}),\nonumber 
\end{align}
for all $v\in L^{2}$. Hence,
\begin{align*}
\E\int_{0}^{t}\vp^{\ve}(X_{r}^{\ve,n})dr & \ge\E\int_{0}^{t}\vp(X_{r}^{\ve,n})dr-C\ve\E\int_{0}^{t}(1+\|X_{r}^{\ve,n}\|_{2}^{2})dr
\end{align*}
and thus
\[
\liminf_{\ve\to0}\E\int_{0}^{t}\vp^{\ve}(X_{r}^{\ve,n})dr\ge\E\int_{0}^{t}\vp(X_{r}^{n})dr.
\]
Using $\vp^{\ve}\le\vp$, due to \eqref{eq:MY-ineq}, and \eqref{eq:SFDE_strong_bound} we may thus let $\ve\to0$ and then $n\to\infty$ in \eqref{eq:svi-2} to obtain 
\begin{align*}
 & \E\|X_{t}-Z_{t}\|_{H^{-1}}^{2}+2\E\int_{0}^{t}\vp(X_{r})dr\\
\le & \E\|x_{0}-Z_{0}\|_{H^{-1}}^{2}+2\E\int_{0}^{t}\vp(Z_{r})dr\\
 & -2\E\int_{0}^{t}(G_{r},X_{r}-Z_{r})_{H^{-1}}dr+C\E\int_{0}^{t}\|X_{r}-Z_{r}\|_{H^{-1}}^{2}dr,
\end{align*}
where $\vp(X)\in L^{1}([0,T]\times\O)$ follows from lower-semicontinuity of $\vp$ on $H^{-1}$.

\textbf{Step 2:} Uniqueness

Let $X$ be an SVI solution to \eqref{eq:SVI-SFDE} and let $Y^{\ve,n}$ be the (strong) solution to \eqref{eq:approx_SVI_constr-1} with initial condition $y_{0}^{n}\in L^{2}(\O;L^{2})$ satisfying $y_{0}^{n}\to y_{0}$ in $L^{2}(\O;H^{-1})$. Then \eqref{eq:SFDE_gen_soln} with $Z=Y^{\ve,n}$ and $G=\ve\D Y^{\ve,n}+\D\phi^{\ve}(Y^{\ve,n})$ yields 
\begin{align}
 & \E\|X_{t}-Y_{t}^{\ve,n}\|_{H^{-1}}^{2}+2\E\int_{0}^{t}\vp(X_{r})dr\nonumber \\
\le & \E\|x_{0}-y_{0}^{n}\|_{H^{-1}}^{2}+2\E\int_{0}^{t}\vp(Y_{r}^{\ve,n})dr\label{eq:approx_SVI}\\
 & -2\E\int_{0}^{t}(\ve\D Y_{r}^{\ve,n}+\D\phi^{\ve}(Y_{r}^{\ve,n}),X_{r}-Y_{r}^{\ve,n})_{H^{-1}}dr\nonumber \\
 & +C\E\int_{0}^{t}\|X_{r}-Y_{r}^{\ve,n}\|_{H^{-1}}^{2}dr.\nonumber 
\end{align}
For $x\in L^{2}$ we have
\[
-(\D\phi^{\ve}(Y^{\ve,n}),x-Y^{\ve,n})_{H^{-1}}+\vp^{\ve}(Y^{\ve,n})\le\vp^{\ve}(x)\quad dt\otimes d\P-\text{a.e.}.
\]
Due to \eqref{eq:SVI_phi_error} we obtain
\[
-(\D\phi^{\ve}(Y^{\ve,n}),x-Y^{\ve,n})_{H^{-1}}+\vp(Y^{\ve,n})\le\vp(x)+C\ve(1+\vp(Y^{\ve,n}))\quad dt\otimes d\P-\text{a.e.}.
\]
Since $\vp(X)\in L^{1}([0,T]\times\O)$ and since $\vp$ is the lower-semicontinuous hull of $\vp_{|L^{2}}$ on $H^{-1}$, for a.e. $(t,\o)\in[0,T]\times\O$ we can find a sequence $x^{m}\in L^{2}$ such that $x^{m}\to X_{t}(\o)$ in $H^{-1}$ and $\vp(x^{m})\to\vp(X_{t}(\o))$. Hence,
\[
-(\D\phi^{\ve}(Y^{\ve,n}),X-Y^{\ve,n})_{H^{-1}}+\vp(Y^{\ve,n})\le\vp(X)+C\ve(1+\vp(Y^{\ve,n}))\quad dt\otimes d\P-\text{a.e.}
\]
and \eqref{eq:approx_SVI} implies
\begin{align*}
\E\|X_{t}-Y_{t}^{\ve,n}\|_{H^{-1}}^{2} & \le\E\|x_{0}-y_{0}^{n}\|_{H^{-1}}^{2}\\
 & +2\E\int_{0}^{t}\ve^{\frac{4}{3}}\|\D Y_{r}^{\ve,n}\|_{H^{-1}}^{2}+\ve^{\frac{2}{3}}\|X_{r}-Y_{r}^{\ve,n}\|_{H^{-1}}^{2}dr\\
 & +C\E\int_{0}^{t}\|X_{r}-Y_{r}^{\ve,n}\|_{H^{-1}}^{2}dr+C\ve\E\int_{0}^{t}(1+\vp(Y_{r}^{\ve,n}))dr.
\end{align*}
Taking $\ve\to0$ then $n\to\infty$ yields
\begin{align*}
\E\|X_{t}-Y_{t}\|_{H^{-1}}^{2}\le & \E\|x_{0}-y_{0}\|_{H^{-1}}^{2}+C\E\int_{0}^{t}\|X_{r}-Y_{r}\|_{H^{-1}}^{2}dr,
\end{align*}
which by Gronwall's inequality concludes the proof. 
\end{proof}

\section{Regularity and Strong solutions\label{sec:Strong-solutions}}

We consider SPDE of the form
\begin{align}
dX_{t} & \in\D(|X_{t}|^{m-1}X_{t})dt+\sum_{k=1}^{\infty}g^{k}X_{t}d\b_{t}^{k},\label{eq:SFDE-mult}\\
X_{0} & =x_{0}\nonumber 
\end{align}
with $m\in[0,1]$ and zero Dirichlet boundary conditions on a smooth, bounded domain $\mcO\subseteq\R^{d}$, in arbitrary dimension $d\in\N$. Here, $\b^{k}$ are independent Brownian motions on a normal, filtered probability space $(\O,\mcF,(\mcF_{t}){}_{t\ge0},\P)$ and
\begin{enumerate}
\item [(B)] $g^{k}\in C^{1}(\bar{\mcO})$ with
\[
\sum_{k=1}^{\infty}\|g^{k}\|_{C^{1}(\bar{\mcO})}^{2}<\infty.
\]

\end{enumerate}
For $v\in H^{-1}$ we set
\[
B(v)(h)=\sum_{k=1}^{\infty}g^{k}v(e_{k},h)_{H^{-1}},
\]
where $e_{k}\in H^{-1}$ is an orthonormal basis of $H^{-1}$. Then $B:H^{-1}\to L_{2}(H^{-1},H^{-1})$ is Lipschitz continuous, i.e.
\begin{align*}
\|B(u)-B(v)\|_{L_{2}(H^{-1},H^{-1})}^{2} & =\sum_{k=1}^{\infty}\|g^{k}(u-v)\|_{H^{-1}}^{2}\\
 & \le\sum_{k=1}^{\infty}\|g^{k}\|_{C^{1}(\bar{\mcO})}^{2}\|u-v\|_{H^{-1}}^{2},\quad\text{for all }u,v\in H^{-1}
\end{align*}
and
\begin{align*}
\|B(v)\|_{L_{2}(H^{-1},L^{2})}^{2} & =\sum_{k=1}^{\infty}\|g^{k}v\|_{L^{2}}^{2}\\
 & \le\sum_{k=1}^{\infty}\|g^{k}\|_{C^{0}(\bar{\mcO})}^{2}\|v\|_{L^{2}}^{2},\quad\text{for all }v\in L^{2}.
\end{align*}
We define $\psi,\phi,\vp$ and $\mcM,L^{m+1}\cap H^{-1},\mcM\cap H^{-1}$ as in Section \ref{sec:SVI} and rewrite \eqref{eq:SFDE-mult} in the \textit{relaxed} form
\begin{align}
dX_{t} & \in-\partial\vp(X_{t})dt+B(X_{t})dW_{t},\label{eq:SFDE_general_form}\\
X_{0} & =x_{0}.\nonumber 
\end{align}
(Generalized) strong solutions to \eqref{eq:SFDE-mult} are then defined as in Definition \ref{def:strong_soln-general} (with $H=H^{-1}$). 

For initial conditions $x_{0}\in L^{2}(\Omega;H^{-1})$ satisfying $\E\vp(x_{0})<\infty$ we prove the existence of strong solutions to \eqref{eq:SFDE-mult}. Moreover, we will prove regularizing properties with respect to the initial condition due to the subgradient structure of the drift. This allows to characterize solutions for initial conditions $x_{0}\in L^{2}(\Omega;H^{-1})$ as generalized strong solutions.
\begin{thm}
\label{thm:FDE}Let $x_{0}\in L^{2}(\Omega;H^{-1})$. 
\begin{enumerate}
\item There is a unique generalized strong solution $(X,\eta)$ to \eqref{eq:SFDE-mult} and $X$ satisfies
\begin{align*}
 & \E t\vp(X_{t})+\E\int_{0}^{t}r\|\eta_{r}\|_{H^{-1}}^{2}dr\le C\left(\E\|x_{0}\|_{H^{-1}}^{2}+1\right).
\end{align*}

\item If \textup{$\E\vp(x_{0})<\infty$,} then there is a unique strong solution $(X,\eta)$ to \eqref{eq:SFDE-mult} satisfying
\[
\E\vp(X_{t})+\E\int_{0}^{t}\|\eta_{r}\|_{H^{-1}}^{2}dr\le\E\vp(x_{0})+C,
\]

\end{enumerate}
The (generalized) strong solution $ $$(X,\eta)$ coincides with the limit solution constructed in \cite{GT11}.
\end{thm}

The proof of Theorem \ref{thm:FDE} proceeds in several steps. In particular, the singularity of $\psi$ causes the need for a non-singular regularization of $\psi$. This approximation has to be carefully chosen in order to obtain uniform bounds. Passing to the limit will in turn rely on Mosco-convergence of the regularized potentials. We will first consider the case of non-singular potentials, keeping careful track of the arising constants in Section \ref{sec:non-singular}, the limit will then be taken in Section \ref{sec:proof}.

\subsubsection{Non-singular potential $\psi$\label{sec:non-singular}}

In this section we restrict to the approximating case of a smooth, non-singular nonlinearity $\psi.$ We assume that $\psi\in C^{3}(\R;\R_{+})$ is convex with Lipschitz continuous derivatives $\phi=\dot{\psi}$, $\dot{\phi}$ satisfying $\psi(0)=\phi(0)=0$ and 
\begin{align}
\phi(r)r & \ge c_{\psi}\psi(r)-C_{\psi}\label{eq:psi-cdt}\\
\dot{\phi}(r)r^{2} & \le C_{\psi}\psi(r),\quad\forall r\in\R,\nonumber 
\end{align}
for some constants $C_{\psi},c_{\psi}>0$.

We consider the following non-degenerate, non-singular approximation of \eqref{eq:SFDE-mult} (cf. Appendix \ref{sec:app-non-deg}): 
\begin{align}
dX_{t}^{\ve} & =\ve\D X_{t}^{\ve}dt+\D\phi(X_{t}^{\ve})dt+B(X_{t}^{\ve})dW_{t},\label{eq:SFDE-non-deg-on-sing}\\
X_{0}^{\ve} & =x_{0}.\nonumber 
\end{align}
For $v\in H^{-1}$ we define
\[
\vp^{\ve}(v):=\begin{cases}
\frac{\ve}{2}\int_{\mcO}|v|^{2}dx+\int_{\mcO}\psi(v)dx, & \text{for }v\in L^{2}\\
+\infty, & \text{otherwise.}
\end{cases}
\]
Note that $\vp^{\ve}\in C^{1}(L^{2})$ with Lipschitz continuous derivative given by
\begin{align*}
D\vp^{\ve}(v)(h) & =\ve\int_{\mcO}vhdx+\int_{\mcO}\phi(v)hdx.
\end{align*}
To check the claimed continuity we note that
\begin{align*}
D\vp^{\ve}(v)(h)-D\vp^{\ve}(w)(h) & =\ve\int_{\mcO}(v-w)hdx+\int_{\mcO}(\phi(v)-\phi(w))hdx\\
 & \le\ve\|v-w\|_{2}\|h\|_{2}+\|\phi\|_{Lip}\|v-w\|_{2}\|h\|_{2}\\
 & \lesssim(\ve+1)\|v-w\|_{2}\|h\|_{2}.
\end{align*}
Moreover, for $k\in\N$ large enough, we have $\vp^{\ve}\in C^{2}(H_{0}^{1}\cap H^{2k+1})$ with Lipschitz continuous second derivative given by
\begin{align*}
D^{2}\vp^{\ve}(v)(g,h) & =\ve\int_{\mcO}hgdx+\int_{\mcO}h\dot{\phi}(v)gdx.
\end{align*}
Indeed:
\begin{align*}
D^{2}\vp^{\ve}(v)(g,h)-D^{2}\vp^{\ve}(w)(g,h) & =\int_{\mcO}h\left(\dot{\phi}(v)-\dot{\phi}(w)\right)gdx\\
 & \le\|h\|_{3}\|\dot{\phi}(v)-\dot{\phi}(w)\|_{3}\|g\|_{3}\\
 & \le\|\dot{\phi}\|_{Lip}\|h\|_{H^{2k+1}}\|v-w\|_{H^{2k+1}}\|g\|_{H^{2k+1}},
\end{align*}
where we used the Sobolev embedding $H^{2k+1}\hookrightarrow L^{3}$ for $k\in\N$ large enough. Moreover, $\vp^{\ve}$ is a convex, lower-semicontinuous function on $H^{-1}$ with subgradient given by
\[
A^{\ve}(v):=-\partial\vp^{\ve}(v)=\ve\D v+\D\phi(v)\in H^{-1},\quad\text{for }v\in H_{0}^{1}.
\]
Hence, we may write \eqref{eq:SFDE-non-deg-on-sing} as 
\begin{align*}
dX_{t}^{\ve} & =-\partial\vp^{\ve}(X_{t}^{\ve})dt+B(X_{t}^{\ve})dW_{t}.
\end{align*}
(Generalized) strong solutions to \eqref{eq:SFDE-non-deg-on-sing} are then defined according to Definition \ref{def:strong_soln-general}. By \cite{PR07}, for each $x_{0}\in L^{2}(\O;H^{-1})$ there is a unique variational solution $X^{\ve}$ to \eqref{eq:SFDE-non-deg-on-sing} with respect to the Gelfand triple
\[
L^{2}\hookrightarrow H^{-1}\hookrightarrow(L^{2})^{*}.
\]
Due to Lemma \ref{lem:fde-l2-test}, $X^{\ve}$ is a strong solution if $x_{0}\in L^{2}(\O;L^{2})$. 
\begin{lem}
\label{lem:H-bound-1}Let $x_{0}\in L^{2}(\O;L^{2})$. For each $\ve>0$ we have $\vp^{\ve}(X^{\ve})\in L^{1}([0,T]\times\Omega)$ with
\[
\E\int_{0}^{T}\vp^{\ve}(X_{r}^{\ve})dr\le C(\E\|x_{0}\|_{H^{-1}}^{2}+1),
\]
for some constant $C$ independent of $\ve>0$ and depending on $\psi$ via the constants $c_{\psi},C_{\psi}$ only. \end{lem}
\begin{proof}
Note that, using \eqref{eq:psi-cdt} 
\begin{align*}
(A^{\ve}(v),v)_{H^{-1}} & =-\int_{\mcO}\left(\ve|v|^{2}+\phi(v)v\right)dx\\
 & \le-\int_{\mcO}\left(\ve|v|^{2}+c\psi(v)-C\right)dx\\
 & \le-c\vp^{\ve}(v)+C,
\end{align*}
for all $v\in H_{0}^{1}$.  By Itô's formula we have
\begin{align*}
 & \E e^{-Kt}\|X_{t}^{\ve}\|_{H^{-1}}^{2}\\
= & \E\|x_{0}\|_{H}^{2}+2\E\int_{0}^{t}e^{-Kr}(A^{\ve}(X_{r}^{\ve}),X_{r}^{\ve})_{H^{-1}}+e^{-Kr}\|B(X_{r}^{\ve})\|_{L_{2}(U,H^{-1})}^{2}dr\\
 & -K\int_{0}^{t}e^{-Kr}\|X_{r}^{\ve}\|_{H^{-1}}^{2}dr\\
\le & \E\|x_{0}\|_{H^{-1}}^{2}-2\E\int_{0}^{t}ce^{-Kr}\vp^{\ve}(X_{r}^{\ve})+Ce^{-Kr}\|X_{r}^{\ve}\|_{H^{-1}}^{2}dr\\
 & -K\int_{0}^{t}e^{-Kr}\|X_{r}^{\ve}\|_{H^{-1}}^{2}dr+C.
\end{align*}
Choosing $K$ large enough yields the claim.
\end{proof}

Based on the strong solution property of $X^{\ve}$ we derive the key estimate in the following
\begin{lem}
\label{lem:strong_approx-2}Let $x_{0}\in L^{2}(\O;L^{2})$. For all $\ve>0$ we have
\begin{align*}
 & \E t\vp^{\ve}(X_{t}^{\ve})+\E\int_{0}^{t}r\|\ve X_{r}^{\ve}+\phi(X_{r}^{\ve})\|_{H_{0}^{1}}^{2}dr\le C\left(\E\|x_{0}\|_{H^{-1}}^{2}+1\right).
\end{align*}
and
\begin{equation}
\E\vp^{\ve}(X_{t}^{\ve})+\E\int_{0}^{t}\|\ve X_{r}^{\ve}+\phi(X_{r}^{\ve})\|_{H_{0}^{1}}^{2}dr\le C\E\vp^{\ve}(x_{0}),\label{approx_strong_S-4}
\end{equation}
for some constant $C$ independent of $\ve>0$ and depending on $\psi$ via the constants $c_{\psi},C_{\psi}$ only.\end{lem}
\begin{proof}
Let $J^{\l}$ be the resolvent of $-\D$ on $H^{-1}$. We let $G^{\l}=G^{\l,k}:=J^{\l}\circ\dots\circ J^{\l}$ the $k$-th iteration of $J^{\l}$. Then $G^{\l}:H^{-1}\to H_{0}^{1}\cap H^{2k+1}$ is linear and continuous. Moreover,
\begin{align*}
\|G^{\l}v\|_{2} & \le\|v\|_{2},\quad\forall v\in L^{2}\\
\|G^{\l}v\|_{H_{0}^{1}} & \le\|v\|_{H_{0}^{1}},\quad\forall v\in H_{0}^{1}
\end{align*}
and
\begin{align*}
\|G^{\l}v-v\|_{2} & \le\|J^{\l}G^{\l,k-1}v-J^{\l}v\|_{2}+\|J^{\l}v-v\|_{2}\\
 & \le\|G^{\l,k-1}v-v\|_{2}+\|J^{\l}v-v\|_{2}.
\end{align*}
Iterating this yields
\begin{align*}
\|G^{\l}v-v\|_{2} & \le k\|J^{\l}v-v\|_{2}\to0\quad\text{for }\l\to0
\end{align*}
for all $v\in L^{2}$. Analogously,
\begin{align*}
\|G^{\l}v-v\|_{H_{0}^{1}} & \to0\quad\text{for }\l\to0
\end{align*}
for all $v\in H_{0}^{1}$.

We define $\vp^{\ve,\l}:=\vp^{\ve}\circ G^{\l}.$ Since $G^{\l}:H^{-1}\to H_{0}^{1}\cap H^{2k+1}$ is a linear, continuous operator we have $\vp^{\ve,\l}\in C^{2}(H^{-1})$ with Lipschitz continuous derivatives given by
\begin{align*}
D\vp^{\ve,\l}(v)(h) & =-\left(\ve\D G^{\l}v+\D\phi(G^{\l}v),G^{\l}h\right)_{H^{-1}},\\
D^{2}\vp^{\ve,\l}(v)(g,h) & =-\ve\int_{\mcO}(G^{\l}h)(G^{\l}g)dx+\int_{\mcO}(G^{\l}h)\dot{\phi}(G^{\l}v)(G^{\l}g)dx.
\end{align*}
By Lemma \ref{lem:fde-l2-test} there is a unique strong solution $X^{\ve}$ to \eqref{eq:SFDE-non-deg-on-sing} with $X^{\ve}\in L^{2}(\O;L^{\infty}([0,T];L^{2}))\cap L^{2}([0,T]\times\O;H_{0}^{1})$. We apply Itô's formula to $t\vp^{\ve,\l}(X_{t}^{\ve})$ to get:
\begin{align*}
 & \E t\vp^{\ve,\l}(X_{t}^{\ve})\\
 & =-\E\int_{0}^{t}r(\ve\D G^{\l}X_{r}^{\ve}+\D\phi(G^{\l}X_{r}^{\ve}),G^{\l}(\ve\D X_{r}^{\ve}+\D\phi(X_{r}^{\ve})))_{H^{-1}}dr\\
 & +\frac{\ve}{2}\sum_{k=1}^{\infty}\E\int_{0}^{t}r\int_{\mcO}|G^{\l}(g^{k}X_{r}^{\ve})|^{2}dxdr\\
 & +\frac{1}{2}\sum_{k=1}^{\infty}\E\int_{0}^{t}r\int_{\mcO}G^{\l}(g^{k}X_{r}^{\ve})\dot{\phi}(G^{\l}X_{r}^{\ve})G^{\l}(g^{k}X_{r}^{\ve})dxdr\\
 & +\E\int_{0}^{t}\vp^{\ve,\l}(X_{r}^{\ve})dr.
\end{align*}
We first note that
\begin{align*}
\int_{\mcO}|G^{\l}(g^{k}X_{r}^{\ve})|^{2}dx & \le\int_{\mcO}|g^{k}X_{r}^{\ve}|^{2}dx\\
 & \le\|g^{k}\|_{C^{0}}^{2}\|X_{r}^{\ve}\|_{2}^{2}.
\end{align*}
Moreover,
\begin{align*}
\int_{\mcO}G^{\l}(g^{k}X_{r}^{\ve})\dot{\phi}(G^{\l}X_{r}^{\ve})G^{\l}(g^{k}X_{r}^{\ve})dx & =\int_{\mcO}(G^{\l}(g^{k}X_{r}^{\ve}))^{2}\dot{\phi}(G^{\l}X_{r}^{\ve})dx\\
 & \le\|\dot{\phi}\|_{\infty}\|g^{k}\|_{C^{0}}^{2}\|X_{r}^{\ve}\|_{2}^{2}.
\end{align*}
We note $G^{\l}v\to v$ in $H_{0}^{1}$ for $v\in H_{0}^{1}$. Since $\dot{\phi}$ is Lipschitz we have $\dot{\phi}(G^{\l}X^{\ve})\to\dot{\phi}(X^{\ve})$ in $L^{2}([0,T]\times\O;H_{0}^{1})$. Moreover, $G^{\l}g^{k}X_{r}^{\ve}\to g^{k}X_{r}^{\ve}$ in $L^{2}([0,T]\times\O;H_{0}^{1})$. Using  \eqref{eq:psi-cdt} this implies
\begin{align*}
 & \lim_{\l}\E\int_{0}^{t}r\int_{\mcO}G^{\l}(g^{k}X_{r}^{\ve})\dot{\phi}(G^{\l}X_{r}^{\ve})G^{\l}(g^{k}X_{r}^{\ve})dxdr\\
 & =\E\int_{0}^{t}r\int_{\mcO}(g^{k}X_{r}^{\ve})^{2}\dot{\phi}(X_{r}^{\ve})dxdr\\
 & \le\|g^{k}\|_{C^{0}}^{2}C_{\psi}\E\int_{0}^{t}r\int_{\mcO}\psi(X_{r}^{\ve})dxdr\\
 & =\|g^{k}\|_{C^{0}}^{2}C_{\psi}\E\int_{0}^{t}r\vp(X_{r}^{\ve})dr.
\end{align*}
Hence, by dominated convergence we obtain
\begin{align*}
 & \lim_{\l\to0}\frac{\ve}{2}\sum_{k=1}^{\infty}\E\int_{0}^{t}r\int_{\mcO}|G^{\l}(g^{k}X_{r}^{\ve})|^{2}dxdr\\
 & +\lim_{\l\to0}\frac{1}{2}\sum_{k=1}^{\infty}\E\int_{0}^{t}r\int_{\mcO}G^{\l}(g^{k}X_{r}^{\ve})\dot{\phi}(G^{\l}(X_{r}^{\ve}))G^{\l}(g^{k}X_{r}^{\ve})dxdr\\
 & \le C\left(1+\E\int_{0}^{t}r\vp^{\ve}(X_{r}^{\ve})dr\right),
\end{align*}
for some constant $C$ independent of $\ve>0$ and depending on $\psi$ via the constants $c_{\psi},C_{\psi}$ only. We note that
\begin{align*}
 & (\ve\D G^{\l}X_{r}^{\ve}+\D\phi(G^{\l}X_{r}^{\ve}),G^{\l}(\ve\D X_{r}^{\ve}+\D\phi(X_{r}^{\ve})))_{H^{-1}}\\
 & =(\ve G^{\l}X_{r}^{\ve}+\phi(G^{\l}X_{r}^{\ve}),G^{\l}(\ve X_{r}^{\ve}+\phi(X_{r}^{\ve})))_{H_{0}^{1}}.
\end{align*}
Since $X^{\ve}\in L^{2}([0,T]\times\Omega;H_{0}^{1})$ we have (using dominated convergence) 
\[
G^{\l}X^{\ve}\to X^{\ve}\text{ in }L^{2}([0,T]\times\Omega;H_{0}^{1})\text{ for }\l\to0.
\]
Since $\phi,\dot{\phi}$ are Lipschitz continuous this implies $\phi(G^{\l}X^{\ve})\to\phi(X^{\ve})$ for $\l\to0$ in $L^{2}([0,T]\times\Omega;H_{0}^{1})$.  Hence, we obtain
\begin{align*}
 & \lim_{\l\to0}-\E\int_{0}^{t}r(\ve G^{\l}X_{r}^{\ve}+\phi(G^{\l}X_{r}^{\ve}),G^{\l}(\ve X_{r}^{\ve}+\phi(X_{r}^{\ve})))_{H_{0}^{1}}dr\\
 & =-\E\int_{0}^{t}r\|\ve X_{r}^{\ve}+\phi(X_{r}^{\ve})\|_{H_{0}^{1}}^{2}dr.
\end{align*}
Since $\vp^{\ve}$ is continuous on $L^{2}$, 
\[
|\vp^{\ve,\l}(v)|\le C(1+\|v\|_{2}^{2})
\]
and $G^{\l}X^{\ve}\to X^{\ve}$ in $L^{2}([0,T]\times\Omega;H_{0}^{1})$, $G^{\l}X_{t}^{\ve}\to X_{t}^{\ve}$ in $L^{2}(\Omega;L^{2})$ for all $t\in[0,T]$, by dominated convergence we get
\begin{align*}
 & \lim_{\l\to0}\E t\vp^{\ve,\l}(X_{t}^{\ve})=\E t\vp^{\ve}(X_{t}^{\ve})
\end{align*}
and 
\[
\lim_{\l\to0}\E\int_{0}^{t}\vp^{\ve,\l}(X_{r}^{\ve})dr=\E\int_{0}^{t}\vp^{\ve}(X_{r}^{\ve})dr.
\]
Putting these estimates together yields
\begin{align*}
\E t\vp^{\ve}(X_{t}^{\ve})\le & -\E\int_{0}^{t}r\|\ve X_{r}^{\ve}+\phi(X_{r}^{\ve})\|_{H_{0}^{1}}^{2}dr+C\left(1+\E\int_{0}^{t}r\vp^{\ve}(X_{r}^{\ve})dr\right)\\
 & +\E\int_{0}^{t}\vp^{\ve}(X_{r}^{\ve})dr.
\end{align*}
By Lemma \ref{lem:H-bound-1} we conclude
\begin{align*}
 & \E t\vp^{\ve}(X_{t}^{\ve})+\E\int_{0}^{t}r\|\ve X_{r}^{\ve}+\phi(X_{r}^{\ve})\|_{H_{0}^{1}}^{2}dr\le C\left(\E\|x_{0}\|_{H^{-1}}^{2}+1\right),
\end{align*}
for some constant $C$ independent of $\ve>0$ and depending on $\psi$ via the constants $c_{\psi},C_{\psi}$ only.

To prove \eqref{approx_strong_S-4} we proceed as above but applying Itô's formula for $\vp^{\ve,\l}(X_{t}^{\ve})$ instead of $t\vp^{\ve,\l}(X_{t}^{\ve})$.
\end{proof}

\subsubsection{Proof of Theorem \ref{thm:FDE}\label{sec:proof}}
\begin{proof}
[Proof of Theorem \ref{thm:FDE}:]\textbf{ }The proof proceeds via a three-step approximation. First, we approximate the singular potential $\psi(r)=\frac{1}{m+1}|r|^{m+1}$ by the smooth functions
\[
\psi^{\d}(r)=\frac{1}{m+1}\left(\left(r^{2}+\d\right)^{\frac{m+1}{2}}-\d^{\frac{m+1}{2}}\right).
\]
Note
\begin{align*}
\phi^{\d}(r): & =(\psi^{\d})'(r)=\left(r^{2}+\d\right)^{\frac{m-1}{2}}r\\
\dot{\phi}^{\d}(r): & =\left(r^{2}+\d\right)^{\frac{m-1}{2}}+(m-1)\left(r^{2}+\d\right)^{\frac{m-3}{2}}r^{2}\\
 & =\left(r^{2}+\d\right)^{\frac{m-3}{2}}(\d+mr^{2}).
\end{align*}
Thus, for $\d\in[0,1]$
\begin{align*}
\phi^{\d}(r)r=\left(r^{2}+\d\right)^{\frac{m-1}{2}}r^{2} & =\left(r^{2}+\d\right)^{\frac{m+1}{2}}-\left(r^{2}+\d\right)^{\frac{m-1}{2}}\d\\
 & \ge(m+1)\psi^{\d}(r)-1
\end{align*}
and
\begin{align*}
\dot{\phi}^{\d}(r)r^{2} & =\left(r^{2}+\d\right)^{\frac{m-3}{2}}(\d+mr^{2})r^{2}\\
 & \le(m+1)\left(r^{2}+\d\right)^{\frac{m+1}{2}}\\
 & =(m+1)^{2}\psi^{\d}(r).
\end{align*}
That is, \eqref{eq:psi-cdt} is satisfied with constants $c_{\psi},C_{\psi}$ independent of $\d>0$. Moreover, we observe
\begin{equation}
|\psi^{\d}(r)-\psi(r)|=\frac{2}{m+1}\d^{\frac{m+1}{2}}.\label{eq:del_convergence}
\end{equation}
Next we consider vanishing viscosity 
\begin{align}
dX_{t}^{\ve,\d} & =\ve\D X_{t}^{\ve,\d}dt+\D\phi^{\d}(X_{t}^{\ve,\d})dt+B(X_{t}^{\ve,\d})dW_{t},\label{eq:strong-eps-delta-approx}\\
X_{0}^{\ve,\d} & =x_{0}.\nonumber 
\end{align}
In a third approximating step we consider smooth approximations of the initial condition, i.e. we first assume \textbf{$x_{0}\in L^{2}(\O;L^{2})$}.

\textbf{Step 1: $\d\to0$}

Let $\ve>0$, $x_{0}\in L^{2}(\O;L^{2})$ and for $v\in H^{-1}$ set 
\[
\vp^{\ve,\d}(v):=\begin{cases}
\frac{\ve}{2}\int_{\mcO}|v|^{2}dx+\int_{\mcO}\psi^{\d}(v)dx, & \text{for }v\in L^{2}\\
+\infty, & \text{otherwise. }
\end{cases}
\]
Since 
\[
\partial\vp^{\ve,\d}(v)=\begin{cases}
-\ve\D v-\D\phi^{\d}(v), & \text{for }v\in H_{0}^{1}\\
\emptyset, & \text{otherwise. }
\end{cases}
\]
we may rewrite \eqref{eq:strong-eps-delta-approx} as
\begin{align*}
dX_{t}^{\ve,\d} & =-\partial\vp^{\ve,\d}(X_{t}^{\ve,\d})dt+B(X_{t}^{\ve,\d})dW_{t}
\end{align*}
and define strong solutions to \eqref{eq:strong-eps-delta-approx} as in Definition \ref{def:strong_soln-general}. By Lemma \ref{lem:fde-l2-test} there is a (unique) strong solution $X^{\ve,\d}$ to \eqref{eq:strong-eps-delta-approx}. From Lemma \ref{lem:fde-l2-test} and Lemma \ref{lem:strong_approx-2} we have
\begin{align}
\E\sup_{t\in[0,T]}\|X_{t}^{\ve,\d}\|_{2}^{2}+2\ve\E\int_{0}^{T}\|X_{r}^{\ve,\d}\|_{H_{0}^{1}}^{2}dr & \le C(\E\|x_{0}\|_{2}^{2}+1)\label{eq:eps_del_bounds_1}\\
\E t\vp^{\ve,\d}(X_{t}^{\ve,\d})+\E\int_{0}^{t}r\|\ve X_{r}^{\ve,\d}+\phi^{\d}(X_{r}^{\ve,\d})\|_{H_{0}^{1}}^{2}dr & \le C\left(\E\|x_{0}\|_{H^{-1}}^{2}+1\right)\nonumber 
\end{align}
and
\begin{align}
\E\vp^{\ve,\d}(X_{t}^{\ve,\d})+\E\int_{0}^{t}\|\ve X_{r}^{\ve,\d}+\phi^{\d}(X_{r}^{\ve,\d})\|_{H_{0}^{1}}^{2}dr & \le C\E\vp^{\ve,\d}(x_{0})\label{eq:eps_del_bounds_2}\\
 & \le C(\E\|x_{0}\|_{2}^{2}+1),\nonumber 
\end{align}
with a constant $C>0$ independent of $\ve,\d>0$. Hence, we may extract a subsequence $\d_{n}\to0$ such that
\begin{align*}
X^{\ve,\d_{n}} & \rightharpoonup\bar{X}^{\ve},\quad\text{in }L^{2}([0,T]\times\O;H_{0}^{1})\\
X^{\ve,\d_{n}} & \rightharpoonup^{*}\bar{X}^{\ve},\quad\text{in }L^{2}(\O;L^{\infty}([0,T];L^{2}))\\
\D\phi^{\d_{n}}(X^{\ve,\d_{n}}) & \rightharpoonup\eta^{\ve},\quad\text{in }L^{2}([0,T]\times\O;H^{-1}).
\end{align*}
Since $B:H^{-1}\to L_{2}(H^{-1};H^{-1})$ is continuous, linear we have $B(X^{\ve,\d_{n}})\rightharpoonup B(\bar{X}^{\ve})$ in $L^{2}([0,T]\times\O;L_{2}(H^{-1};H^{-1}))$ and thus $dt\otimes d\P$-almost everywhere 
\[
\bar{X}_{t}^{\ve}=x_{0}+\int_{0}^{t}(\ve\D\bar{X}_{r}^{\ve}+\eta_{r}^{\ve})dr+\int_{0}^{t}B(\bar{X}_{r}^{\ve})dW_{r}.
\]
Hence, defining
\[
X_{t}^{\ve}:=x_{0}+\int_{0}^{t}(\ve\D\bar{X}_{r}^{\ve}+\eta_{r}^{\ve})dr+\int_{0}^{t}B(\bar{X}_{r}^{\ve})dW_{r}
\]
we have $X^{\ve}=\bar{X}^{\ve}$, $dt\otimes d\P$-almost everywhere and due to \cite[Theorem 4.2.5]{PR07} we have $X^{\ve}\in L^{2}(\O;C([0,T];H^{-1}))$. We aim to prove that $X^{\ve}$ is a strong solution to 
\begin{align}
dX_{t}^{\ve} & \in-\partial\vp^{\ve}(X_{t}^{\ve})dt+B(X_{t}^{\ve})dW_{t},\label{eq:strong-eps-approx}
\end{align}
where (for $v\in H^{-1}$)
\[
\vp^{\ve}(v):=\begin{cases}
\frac{\ve}{2}\int_{\mcO}|v|^{2}dx+\int_{\mcO}\psi(v)dx, & \text{for }v\in L^{2}\\
+\infty, & \text{otherwise. }
\end{cases}
\]
Note that
\begin{align}
\partial\vp^{\ve}(v) & =\begin{cases}
\{\ve\D v+\D\z|\ \z\in H_{0}^{1}\text{ with }\z\in\phi(v)\text{ a.e.}\}, & \text{for }v\in H_{0}^{1}\\
\emptyset, & \text{otherwise. }
\end{cases}\label{eq:eps_subgradient}
\end{align}
It thus remains to identify $\ve\D X^{\ve}+\eta^{\ve}\in-\partial\vp^{\ve}(X^{\ve})$, $dt\otimes d\P$-almost everywhere. Itô's formula yields
\begin{align}
\E e^{-Ct}\|X_{t}^{\ve}\|_{H^{-1}}^{2}= & \E\|x_{0}\|_{H^{-1}}^{2}+\E\int_{0}^{t}e^{-Cr}(\ve\D X_{r}^{\ve}+\eta_{r}^{\ve},X_{r}^{\ve})_{H^{-1}}dr\nonumber \\
 & +\E\int_{0}^{t}e^{-Cr}\|B(X_{r}^{\ve})\|_{L_{2}}^{2}dr-C\E\int_{0}^{t}e^{-Cr}\|X_{r}^{\ve}\|_{H^{-1}}^{2}dr\quad\forall t\in[0,T].\label{eq:ito_limit-1}
\end{align}
Using Itô's formula for $X^{\ve,\d}$ yields (for $C>0$ large enough) 
\begin{align*}
 & \E e^{-Ct}\|X_{t}^{\ve,\d_{n}}\|_{H^{-1}}^{2}\\
= & \E\|x_{0}\|_{H^{-1}}^{2}+\E\int_{0}^{t}e^{-Cr}(\ve\D X_{r}^{\ve,\d_{n}}+\D\phi^{\d_{n}}(X_{r}^{\ve,\d_{n}}),X_{r}^{\ve,\d_{n}})_{H^{-1}}dr\\
 & +\E\int_{0}^{t}e^{-Cr}\|B(X_{r}^{\ve,\d_{n}})-B(X_{r}^{\ve})\|_{L_{2}}^{2}dr-\E\int_{0}^{t}e^{-Cr}\|B(X_{r}^{\ve})\|_{L_{2}}^{2}dr\\
 & +2\E\int_{0}^{t}e^{-Cr}(B(X_{r}^{\ve,\d_{n}}),B(X_{r}^{\ve}))_{L_{2}}dr-C\E\int_{0}^{t}e^{-Cr}\|X_{r}^{\ve,\d_{n}}-X_{r}^{\ve}\|_{H^{-1}}^{2}dr\\
 & +C\E\int_{0}^{t}e^{-Cr}\|X_{r}^{\ve}\|_{H^{-1}}^{2}dr-2C\E\int_{0}^{t}e^{-Cr}(X_{r}^{\ve,\d_{n}},X_{r}^{\ve})_{H^{-1}}dr\\
\le & \E\|x_{0}\|_{H^{-1}}^{2}+\E\int_{0}^{t}e^{-Cr}(\ve\D X_{r}^{\ve,\d_{n}}+\D\phi^{\d_{n}}(X_{r}^{\ve,\d_{n}}),X_{r}^{\ve,\d_{n}})_{H^{-1}}dr\\
 & -\E\int_{0}^{t}e^{-Cr}\|B(X_{r}^{\ve})\|_{L_{2}}^{2}dr+2\E\int_{0}^{t}e^{-Cr}(B(X_{r}^{\ve,\d_{n}}),B(X_{r}^{\ve}))_{L_{2}}dr\\
 & +C\E\int_{0}^{t}e^{-Cr}\|X_{r}^{\ve}\|_{H^{-1}}^{2}dr-2C\E\int_{0}^{t}e^{-Cr}(X_{r}^{\ve,\d_{n}},X_{r}^{\ve})_{H^{-1}}dr.
\end{align*}
Taking $\liminf_{\d_{n}\to0}$ we obtain (first in distributional sense in $t$ then a.e. by the Lebesgue Theorem)
\begin{align*}
 & \E e^{-Ct}\|X_{t}^{\ve}\|_{H^{-1}}^{2}\\
\le & \liminf_{\d_{n}\to0}\E e^{-Ct}\|X_{t}^{\ve,\d_{n}}\|_{H^{-1}}^{2}\\
\le & \E\|x_{0}\|_{H^{-1}}^{2}+\liminf_{\d_{n}\to0}\E\int_{0}^{t}e^{-Cr}(\ve\D X_{r}^{\ve,\d_{n}}+\D\phi^{\d_{n}}(X_{r}^{\ve,\d_{n}}),X_{r}^{\ve,\d_{n}})_{H^{-1}}dr\\
 & -\E\int_{0}^{t}e^{-Cr}\|B(X_{r}^{\ve})\|_{L_{2}}^{2}dr+2\E\int_{0}^{t}e^{-Cr}(B(X_{r}^{\ve}),B(X_{r}^{\ve}))_{L_{2}}dr\\
 & +C\E\int_{0}^{t}e^{-Cr}\|X_{r}^{\ve}\|_{H^{-1}}^{2}dr-2C\E\int_{0}^{t}e^{-Cr}(X_{r}^{\ve},X_{r}^{\ve})_{H^{-1}}dr\\
\le & \E\|x_{0}\|_{H^{-1}}^{2}+\liminf_{\d\to0}\E\int_{0}^{t}e^{-Cr}(\ve\D X_{r}^{\ve,\d_{n}}+\D\phi^{\d_{n}}(X_{r}^{\ve,\d_{n}}),X_{r}^{\ve,\d_{n}})_{H^{-1}}dr\\
 & +\E\int_{0}^{t}e^{-Cr}\|B(X_{r}^{\ve})\|_{L_{2}}^{2}dr-C\E\int_{0}^{t}e^{-Cr}\|X_{r}^{\ve}\|_{H^{-1}}^{2}dr\quad\text{a.e. }t\in[0,T].
\end{align*}
Subtracting \eqref{eq:ito_limit-1} we obtain
\begin{align}
 & \E\int_{0}^{t}e^{-Cr}(\ve\D X_{r}^{\ve}+\eta_{r}^{\ve},X_{r}^{\ve})_{H^{-1}}dr\label{eq:minty-1}\\
\le & \liminf_{\d_{n}\to0}\E\int_{0}^{t}e^{-Cr}(\ve\D X_{r}^{\ve,\d_{n}}+\D\phi^{\d_{n}}(X_{r}^{\ve,\d_{n}}),X_{r}^{\ve,\d_{n}})_{H^{-1}}dr.\nonumber 
\end{align}
We now consider the convex, lower-semicontinuous functionals $\bar{\vp}^{\ve},\bar{\vp}^{\ve,\d}:L^{2}([0,T]\times\O;H^{-1})\to\bar{\R}$ defined by
\begin{align*}
\bar{\vp}^{\ve}(v):= & \begin{cases}
\E\int_{0}^{T}e^{-Cr}\int_{\mcO}(\frac{\ve}{2}v^{2}+\psi(v))dxdr & ,\text{ if }v\in L^{2}([0,T]\times\O;L^{2})\\
+\infty & ,\text{ otherwise}
\end{cases}\\
= & \E\int_{0}^{T}e^{-Cr}\vp^{\ve}(v_{r})dr.
\end{align*}
and $\bar{\vp}^{\ve,\d}$ being defined analogously, where we endow $L^{2}([0,T]\times\O;H^{-1})$ with the equivalent norm
\[
\|v\|_{L^{2}([0,T]\times\O;H^{-1})}^{2}:=\E\int_{0}^{T}e^{-Cr}\|v_{r}\|_{H^{-1}}^{2}dr.
\]
Due to the characterization of subgradients of integral functionals proved in \cite[Theorem 21]{R74} we have
\begin{align}
\partial\bar{\vp}^{\ve}(v) & =\{\eta\in L^{2}([0,T]\times\O;H^{-1})|\ \eta\in\partial\vp^{\ve}(v),\ dt\otimes d\P\text{-a.e.}\}\label{eq:eps_subgr}
\end{align}
and 
\begin{equation}
\partial\bar{\vp}^{\ve,\d}(v)=\{-\ve\D v-\D\phi^{\d}(v)\}\quad\text{for }v\in L^{2}([0,T]\times\O;H_{0}^{1}).\label{eq:eps_del_subgr}
\end{equation}
Since $ $$\bar{\vp}^{\ve,\d}\to\bar{\vp}^{\ve}$ uniformly for $\d\to0$ (cf. \eqref{eq:del_convergence}), we also have $\bar{\vp}^{\ve,\d}\to\bar{\vp}^{\ve}$ in Mosco sense. Due to \eqref{eq:eps_del_subgr} we have
\[
(-\ve\D X^{\ve,\d_{n}}-\D\phi^{\d_{n}}(X^{\ve,\d_{n}}),Y-X^{\ve,\d_{n}})_{L^{2}([0,T]\times\O;H^{-1})}+\bar{\vp}^{\ve,\d_{n}}(X^{\ve,\d_{n}})\le\bar{\vp}^{\ve,\d_{n}}(Y),
\]
for all $Y\in L^{2}([0,T]\times\O;H^{-1})$. Using \eqref{eq:minty-1} and Mosco convergence of $\bar{\vp}^{\ve,\d}$ to $\bar{\vp}^{\ve}$ we may take the $\liminf_{n\to\infty}$ to get
\[
(-\ve\D X^{\ve}-\eta^{\ve},Y-X^{\ve})_{L^{2}([0,T]\times\O;H^{-1})}+\bar{\vp}^{\ve}(X^{\ve})\le\bar{\vp}^{\ve}(Y).
\]
Hence, $\ve\D X^{\ve}+\eta^{\ve}\in-\partial\bar{\vp}^{\ve}(X^{\ve})$ and we conclude $\ve\D X^{\ve}+\eta^{\ve}\in-\partial\vp^{\ve}(X^{\ve})$ $dt\otimes d\P$-almost everywhere due to \eqref{eq:eps_subgr}. Then, \eqref{eq:eps_subgradient} yields
\begin{equation}
\eta^{\ve}=\D\z^{\ve}\label{eq:eta_eps_char}
\end{equation}
with $\z^{\ve}\in H_{0}^{1}$ and $\z^{\ve}\in\phi(X^{\ve})$ a.e..  In conclusion, $X^{\ve}$ is a strong solution to \eqref{eq:strong-eps-approx}. Passing to the limit in \eqref{eq:eps_del_bounds_1}, \eqref{eq:eps_del_bounds_2} yields
\begin{align}
\E\sup_{t\in[0,T]}\|X_{t}^{\ve}\|_{2}^{2}+2\ve\E\int_{0}^{T}\|X_{r}^{\ve}\|_{H_{0}^{1}}^{2}dr & \le C(\E\|x_{0}\|_{2}^{2}+1)\label{eq:eps_bounds_1}\\
\E t\vp^{\ve}(X_{t}^{\ve})+\E\int_{0}^{t}r\|\ve\D X^{\ve}+\eta_{r}^{\ve}\|_{H^{-1}}^{2}dr & \le C\left(\E\|x_{0}\|_{H^{-1}}^{2}+1\right)\nonumber 
\end{align}
and
\begin{align}
\E\vp^{\ve}(X_{t}^{\ve})+\E\int_{0}^{t}\|\ve\D X^{\ve}+\eta_{r}^{\ve}\|_{H^{-1}}^{2}dr & \le\E\vp^{\ve}(x_{0})\label{eq:eps_bounds_2}\\
 & \le C(\E\|x_{0}\|_{2}^{2}+1).\nonumber 
\end{align}

\textbf{Step 2: $\ve\to0$}

For $\ve_{1},\ve_{2}>0$ let $(X^{\ve_{1}},\eta^{\ve_{1}})$, $(X^{\ve_{2}},\eta^{\ve_{2}})$ be two strong solutions to \eqref{eq:strong-eps-approx} with initial conditions $x_{0}^{1},x_{0}^{2}\in L^{2}(\O;L^{2})$ respectively. Itô's formula implies 
\begin{align*}
 & e^{-Kt}\|X_{t}^{\ve_{1}}-X_{t}^{\ve_{2}}\|_{H^{-1}}^{2}\\
 & =\|x_{0}^{1}-x_{0}^{2}\|_{H^{-1}}^{2}+\int_{0}^{t}2e^{-Kr}(\ve_{1}\D X^{\ve_{1}}+\eta^{\ve_{1}}-(\ve_{2}\D X^{\ve_{2}}+\eta^{\ve_{2}}),X_{r}^{\ve_{1}}-X_{r}^{\ve_{2}})_{H^{-1}}dr\\
 & +\int_{0}^{t}e^{-Kr}(X_{r}^{\ve_{1}}-X_{r}^{\ve_{2}},B(X_{r}^{\ve_{1}})-B(X_{r}^{\ve_{2}}))_{H^{-1}}dW_{r}\\
 & +\int_{0}^{t}e^{-Kr}\|B(X_{r}^{\ve_{1}})-B(X_{r}^{\ve_{2}})\|_{L_{2}}^{2}dr-K\E\int_{0}^{t}e^{-Kr}\|X_{r}^{\ve_{1}}-X_{r}^{\ve_{2}}\|_{H^{-1}}^{2}dr.
\end{align*}
Due to \eqref{eq:eta_eps_char} we have
\[
(\eta^{\ve_{1}}-\eta^{\ve_{2}},X^{\ve_{1}}-X^{\ve_{2}})_{H^{-1}}\le0
\]
and we note that
\[
(\ve_{1}\D X_{r}^{\ve_{1}}-\ve_{2}\D X_{r}^{\ve_{2}},X_{r}^{\ve_{1}}-X_{r}^{\ve_{2}})_{H^{-1}}\le2(\ve_{1}+\ve_{2})(\|X_{r}^{\ve_{1}}\|_{2}^{2}+\|X_{r}^{\ve_{2}}\|_{2}^{2}).
\]
Hence, the Burkholder-Davis-Gundy inequality and Lemma \ref{lem:fde-l2-test} imply
\begin{align}
\E\sup_{t\in[0,T]}\|X_{t}^{\ve_{1}}-X_{t}^{\ve_{2}}\|_{H^{-1}}^{2}\le & C\E\|x_{0}^{1}-x_{0}^{2}\|_{H^{-1}}^{2}\label{eq:vanishing-ineq}\\
 & +C(\ve_{1}+\ve_{2})\left(\E\|x_{0}^{1}\|_{2}^{2}+\E\|x_{0}^{2}\|_{2}^{2}+1\right).\nonumber 
\end{align}
Let now $x_{0}\in L^{2}(\O;L^{2})$ and for each $\ve>0$ let $(X^{\ve},\eta^{\ve})$ be a solution to \eqref{eq:strong-eps-approx} with initial condition $x_{0}.$ Due to \eqref{eq:vanishing-ineq} there is an $\mcF_{t}$-adapted process $X\in L^{2}([0,T]\times\O;H^{-1})$ with $X_{0}=x_{0}$ such that
\[
X^{\ve}\to X\quad\text{in }L^{2}(\O;C([0,T];H^{-1}))\quad\text{for }\ve\to0.
\]
Using \eqref{eq:eps_bounds_2} we can extract a weakly convergent subsequence 
\[
\ve_{n}\D X^{\ve_{n}}+\eta^{\ve_{n}}\rightharpoonup\eta,\quad\text{in }L^{2}([0,T]\times\O;H^{-1}).
\]
By step one we have $\ve_{n}\D X^{\ve_{n}}+\eta^{\ve_{n}}\in-\partial\vp^{\ve_{n}}(X^{\ve_{n}})$ a.e., hence
\[
(\ve_{n}\D X^{\ve_{n}}+\eta^{\ve_{n}},X^{\ve_{n}}-Y)_{L^{2}([0,T]\times\O;H^{-1})}+\bar{\vp}^{\ve_{n}}(X^{\ve_{n}})\le\bar{\vp}^{\ve_{n}}(Y),
\]
for all $Y\in L^{2}([0,T]\times\O;H^{-1})$. For $v\in L^{2}([0,T]\times\O;H^{-1})$ we define
\[
\bar{\vp}(v)=\E\int_{0}^{T}\vp(v_{r})dr.
\]
Again, due to \cite[Theorem 21]{R74} we have
\begin{equation}
\partial\bar{\vp}(v)=\{\eta\in L^{2}([0,T]\times\O;H^{-1})|\ \eta\in\partial\vp(v),\ dt\otimes d\P\text{-a.e.}\}.\label{eq:subgr}
\end{equation}
For $v\in L^{2}$ we observe that $\vp^{\ve}(v)=\frac{\ve}{2}\|v\|_{2}^{2}+\vp(v)$ and thus
\begin{align}
 & (\ve_{n}\D X^{\ve_{n}}+\eta^{\ve_{n}},X^{\ve_{n}}-Y)_{L^{2}([0,T]\times\O;H^{-1})}+\bar{\vp}(X^{\ve_{n}})\label{eq:approx_subgr}\\
 & \le\bar{\vp}(Y)+C\ve_{n}\|Y\|_{L^{2}([0,T]\times\O;L^{2})}^{2},\nonumber 
\end{align}
for all $Y\in L^{2}([0,T]\times\O;L^{2})$. 

Let $J^{\l}=(1-\l\D)^{-1}$ be the resolvent of $-\D$ on $H^{-1}$. Then
\[
\|J^{\l}v\|_{H^{-1}}\le\|v\|_{H^{-1}}
\]
and $J^{\l}v\to v$ in $H^{-1}$ for $\l\to0$. Moreover,
\[
\|J^{\l}v\|_{L^{2}}\le\frac{C}{\l}\|v\|_{H^{-1}}
\]
for all $v\in H^{-1}$. For $v\in L^{m+1}\cap H^{-1}$ we have 
\[
\vp(J^{\l}v)=\|J^{\l}v\|_{L^{m+1}}\le\|v\|_{L^{m+1}}=\vp(v).
\]
For the case $m=0$, in addition: Let $v=v_{\mu}\in\mcM\cap H^{-1}$. Since $\vp$ is the lower-semicontinuous hull of $\vp$ restricted to $L^{1}\cap H^{-1}$ (cf. Appendix \ref{sec:heaviside_relaxation}), there is a sequence $v^{n}\in L^{1}\cap H^{-1}$ such that $v^{n}\to v$ in $H^{-1}$ and $\vp(v^{n})\to\vp(v)$. By lower-semicontinuity of $\vp$ we conclude
\begin{align}
\vp(J^{\l}v) & \le\lim_{n\to\infty}\vp(J^{\l}v^{n})\nonumber \\
 & \le\lim_{n\to\infty}\vp(v^{n})\label{eq:resovlent_nonexp-1}\\
 & =\vp(v),\quad\forall v\in\mcM\cap H^{-1}.\nonumber 
\end{align}
Given $ $$Y\in L^{2}([0,T]\times\O;H^{-1})$ we set $Y^{\ve}:=J^{\ve^{\frac{1}{4}}}Y\in L^{2}([0,T]\times\O;L^{2})$. By dominated convergence $Y^{\ve}\to Y$ in $L^{2}([0,T]\times\O;H^{-1})$. From \eqref{eq:approx_subgr} we obtain 
\begin{align*}
 & (\ve_{n}\D X^{\ve_{n}}+\eta^{\ve_{n}},X^{\ve_{n}}-Y^{\ve_{n}})_{L^{2}([0,T]\times\O;H^{-1})}+\bar{\vp}(X^{\ve_{n}})\\
 & \le\bar{\vp}(Y^{\ve_{n}})+C\ve_{n}\|Y^{\ve_{n}}\|_{L^{2}([0,T]\times\O;L^{2})}^{2}\\
 & \le\bar{\vp}(Y)+C\sqrt{\ve_{n}}\|Y\|_{L^{2}([0,T]\times\O;H^{-1})}^{2}.
\end{align*}
Taking $n\to\infty$ and using lower semicontinuity of $\bar{\vp}$ we arrive at 
\[
(\eta,X-Y)_{L^{2}([0,T]\times\O;H^{-1})}+\bar{\vp}(X)\le\bar{\vp}(Y)
\]
for all $Y\in L^{2}([0,T]\times\O;H^{-1})$ and thus $\eta\in\partial\bar{\vp}(X)$, which implies $\eta\in\partial\vp(X)$ a.e. by \eqref{eq:subgr}. In conclusion, $X$ is a strong solution to 
\begin{align}
dX_{t} & \in-\partial\vp(X_{t})dt+B(X_{t})dW_{t}.\label{eq:limit}
\end{align}
Taking the limit in \eqref{eq:eps_bounds_1}, \eqref{eq:eps_bounds_2} yields
\begin{align}
\E\sup_{t\in[0,T]}\|X_{t}\|_{2}^{2} & \le C(\E\|x_{0}\|_{2}^{2}+1)\label{eq:bounds_1}\\
\E t\vp(X_{t})+\E\int_{0}^{t}r\|\eta_{r}\|_{H^{-1}}^{2}dr & \le C\left(\E\|x_{0}\|_{H^{-1}}^{2}+1\right)\nonumber 
\end{align}
and
\begin{align}
\E\vp(X_{t})+\E\int_{0}^{t}\|\eta_{r}\|_{H^{-1}}^{2}dr & \le C\E\vp(x_{0})\label{eq:bounds_2}\\
 & \le C(\E\|x_{0}\|_{2}^{2}+1).\nonumber 
\end{align}
Moreover, from \eqref{eq:vanishing-ineq} we obtain
\begin{align}
\E\sup_{t\in[0,T]}\|X_{t}^{1}-X_{t}^{2}\|_{H^{-1}}^{2}\le & C\E\|x_{0}^{1}-x_{0}^{2}\|_{H^{-1}}^{2},\label{eq:vanishing-ineq-1}
\end{align}
where $X^{1},X^{2}$ are the corresponding limits for the initial conditions $x_{0}^{1},x_{0}^{2}$ respectively.

\textbf{Step 3: Proof of (i)}

Suppose $x_{0}\in L^{2}(\O;H^{-1})$ satisfying $\E\vp(x_{0})<\infty.$ We consider the case $m=0$, the case $m>0$ can be treated analogously. Let $J^{\l}=(1-\l\D)^{-1}$ be the resolvent of $-\D$ on $H^{-1}$, set $x_{0}^{n}=J^{\frac{1}{n}}x_{0}$ and let $(X^{n},\eta^{n})$ be the corresponding strong solution to \eqref{eq:limit} constructed in step two. By dominated convergence we have
\[
x_{0}^{n}\to x_{0},\quad\text{in }L^{2}(\O;H^{-1}).
\]
Due to \eqref{eq:bounds_2}, \eqref{eq:vanishing-ineq-1} and \eqref{eq:resovlent_nonexp-1} we may extract (weakly) convergent subsequences
\begin{align*}
X^{n} & \to X,\quad\text{in }L^{2}(\O;C([0,T];H^{-1}))\\
\eta^{n} & \rightharpoonup\eta,\quad\text{in }L^{2}([0,T]\times\O;H^{-1}).
\end{align*}
Since $\eta^{n}\in-\partial\bar{\vp}(X^{n})$, strong-weak closedness of the subgradient $\partial\bar{\vp}$ implies $\eta\in-\partial\bar{\vp}(X)$ and thus $\eta\in-\partial\vp(X)$ a.e.. It then easily follows that $(X,\eta)$ is a strong solution to \eqref{eq:SFDE-mult}.

\textbf{Step 3: Proof of (ii) }

Let\textbf{ }$x_{0}\in L^{2}(\O;H^{-1})$ and \textbf{$x_{0}^{n}\in L^{2}(\O;L^{2})$ }with $x_{0}^{n}\to x$ in $L^{2}(\O;H^{-1})$, $\E\|x_{0}^{n}\|_{H^{-1}}^{2}\le\E\|x_{0}\|_{H^{-1}}^{2}$ and let $(X^{n},\eta^{n})$ be the corresponding strong solutions to \eqref{eq:limit} constructed in step two. By \eqref{eq:vanishing-ineq-1} we have
\[
\E\sup_{t\in[0,T]}\|X^{n}-X^{m}\|_{H^{-1}}^{2}\le C\E\|x_{n}-x_{m}\|_{H^{-1}}^{2}.
\]
Hence, $X^{n}\to X$ in $L^{2}(\O;C([0,T];H^{-1}))$. Moreover, 
\begin{align*}
 & \E t\vp(X_{t}^{n})+\E\int_{0}^{t}r\|\eta_{r}^{n}\|_{H^{-1}}^{2}dr\le C\left(\E\|x_{0}\|_{H}^{2}+1\right).
\end{align*}
Hence, there is a map $\eta$ with $\eta\in L^{2}([\tau,T]\times\O;H_{0}^{1})$ such that
\[
\eta^{n}\rightharpoonup\eta,\quad\text{in }L^{2}([\tau,T]\times\O;H^{-1}),
\]
for all $\tau>0$. By strong-weak closedness of $\partial\bar{\vp}$ we have $\eta\in-\partial\bar{\vp}(X)$ and thus $\eta\in-\partial\vp(X)$ a.e.. Hence, $X$ is a generalized strong solution satisfying
\begin{align*}
 & \E t\vp(X_{t})+\E\int_{0}^{t}r\|\eta\|_{H^{-1}}^{2}dr\le C\left(\E\|x_{0}\|_{H^{-1}}^{2}+1\right).
\end{align*}

\end{proof}
\appendix

\section{(Generalized) strong solutions to gradient type SPDE\label{sec:strong_solutions}}

Let $\vp:H\to\R$ be a proper, lower-semicontinuous, convex function on a separable real Hilbert space $H$. We consider SPDE of the type 
\begin{align}
dX_{t} & \in-\partial\vp(X_{t})dt+B(t,X_{t})dW_{t},\label{eq:gradient_SPDE_general}\\
X_{0} & =x_{0},\nonumber 
\end{align}
where $W$ is a cylindrical Wiener process in a separable Hilbert space $U$ defined on a probability space $(\Omega,\mcF,\P)$ with normal filtration $(\mcF_{t})_{t\ge0}$ and $B:[0,T]\times H\times\O\to L_{2}(U,H)$ is Lipschitz continuous, i.e.
\[
\|B(t,v)-B(t,w)\|_{L_{2}(U,H)}^{2}\le C\|v-w\|_{H}^{2}\quad\forall v,w\in H
\]
and all $(t,\o)\in[0,T]\times\O$. Furthermore, we assume that 
\[
\|B(\cdot,0)\|_{L_{2}(U,H)}\in L^{2}([0,T]\times\Omega).
\]
We then define
\begin{defn}
\label{def:strong_soln-general}Let $x_{0}\in L^{2}(\Omega;H).$ An $H$-continuous, $\mcF_{t}$-adapted process $X\in L^{2}(\Omega;C([0,T];H))$ for which there exists a selection $\eta\in-\partial\vp(X)$, $dt\otimes d\P$-a.e. is said to be a 
\begin{enumerate}
\item strong solution to \eqref{eq:gradient_SPDE_general} if 
\[
\eta\in L^{2}([0,T]\times\Omega;H)
\]
and $\P$-a.s.
\[
X_{t}=x_{0}+\int_{0}^{t}\eta_{r}dr+\int_{0}^{t}B(r,X_{r})dW_{r},\quad\forall t\in[0,T].
\]

\item generalized strong solution to \eqref{eq:gradient_SPDE_general} if 
\[
\eta\in L^{2}([\tau,T]\times\Omega;H),\quad\forall\tau>0
\]
and $\P$-a.s.
\[
X_{t}=X_{\tau}+\int_{\tau}^{t}\eta_{r}dr+\int_{\tau}^{t}B(r,X_{r})dW_{r},\quad\forall t\in[\tau,T],
\]
for all $\tau>0$.
\end{enumerate}
\end{defn}

\section{Non-degenerate, non-singular stochastic fast diffusion equations\label{sec:app-non-deg}}

In this section we consider non-degenerate, non-singular approximations to \eqref{eq:SFDE}, that is 
\begin{align}
dX_{t} & =\ve\D X_{t}dt+\D\phi(X_{t})dt+B(t,X_{t})dW_{t},\label{eq:SFDE-1}\\
X_{0} & =x_{0},\nonumber 
\end{align}
where $\phi:\R\to\R$ is a Lipschitz continuous, monotone function satisfying $\phi(0)=0$. We further assume that $W$ is a cylindrical Wiener process in a separable Hilbert space $U$ defined on a probability space $(\Omega,\mcF,\P)$ with normal filtration $(\mcF_{t})_{t\ge0}$ and $B:[0,T]\times H^{-1}\times\O\to L_{2}(U,H^{-1})$ is Lipschitz continuous, i.e.
\begin{align*}
\|B(t,v)-B(t,w)\|_{L_{2}(U,H^{-1})}^{2} & \le C\|v-w\|_{H^{-1}}^{2}\quad\forall v,w\in H^{-1},
\end{align*}
for some constant $C>0$ and all $(t,\o)\in[0,T]\times\O$. We assume
\begin{equation}
\|B(t,v)\|_{L_{2}(U,L^{2})}^{2}\le C(1+\|v\|_{2}^{2}),\quad\forall v\in L^{2},\label{eq:sublinear_B_in_S}
\end{equation}
and all $(t,\o)\in[0,T]\times\O$. By \cite{PR07} there is a unique variational solution $X$ to \eqref{eq:SFDE-1} with respect to the Gelfand triple
\[
L^{2}\hookrightarrow H^{-1}\hookrightarrow(L^{2})^{*}.
\]
Under an additional regularity assumption on the diffusion coefficients $B$ we prove that in fact, these solutions are strong solutions in $H^{-1}$. 
\begin{lem}
\label{lem:fde-l2-test}Let $x_{0}\in L^{2}(\O;L^{2})$. Then 
\[
\E\sup_{t\in[0,T]}\|X_{t}\|_{2}^{2}+\ve\E\int_{0}^{T}\|X_{r}\|_{H_{0}^{1}}^{2}dr\le C(\E\|x_{0}\|_{2}^{2}+1),
\]
with a constant $C>0$ independent of $\ve$ and $\phi$.\end{lem}
\begin{proof}
In the following we let $(e_{i})_{i=1}^{\infty}$ be an orthonormal basis of eigenvectors of $-\D$ in $H^{-1}$. We further let $P^{n}:H^{-1}\to\text{span}\{e_{1},\dots,e_{n}\}$ be the orthogonal projection onto the span of the first $n$ eigenvectors. We recall that the unique variational solution $X$ to \eqref{eq:SFDE-1} is constructed in \cite{PR07} as a (weak) limit in $L^{2}([0,T]\times\O;L^{2})$ of the solutions to the following Galerkin approximation
\begin{align*}
dX_{t}^{n} & =\ve P^{n}\D X_{t}^{n}dt+P^{n}\D\phi(X_{t}^{n})dt+P^{n}B(t,X_{t}^{n})dW_{t}^{n},\\
X_{0}^{n} & =P^{n}x_{0}.
\end{align*}
We observe that
\begin{align*}
\|X_{t}^{n}\|_{2}^{2} & =\|P^{n}x_{0}\|_{2}^{2}+2\int_{0}^{t}(X_{r}^{n},\ve P^{n}\D X_{r}^{n}+P^{n}\D\phi(X_{r}^{n}))_{2}dr\\
 & +2\int_{0}^{t}(X_{r}^{n},P^{n}B(r,X_{r}^{n})dW_{r}^{n})_{2}dr+\int_{0}^{t}\|P^{n}B(r,X_{r}^{n})\|_{L_{2}(U,L^{2})}^{2}dr\\
 & =\|P^{n}x_{0}\|_{2}^{2}-2\ve\int_{0}^{t}\|X_{r}^{n}\|_{H_{0}^{1}}^{2}dr+2\int_{0}^{t}{}_{H_{0}^{1}}\<X_{r}^{n},\D\phi(X_{r}^{n})\>_{H^{-1}}dr\\
 & +2\int_{0}^{t}(X_{r}^{n},P^{n}B(r,X_{r}^{n})dW_{r}^{n})_{2}dr+\int_{0}^{t}\|P^{n}B(r,X_{r}^{n})\|_{L_{2}(U,L^{2})}^{2}dr.
\end{align*}
Since $\phi$ can be approximated by a sequence of increasing, Lipschitz functions in $C^{1}(\R)$ we observe that 
\[
_{H_{0}^{1}}\<X_{r}^{n},\D\phi(X_{r}^{n})\>_{H^{-1}}=-(\nabla X_{r}^{n},\nabla\phi(X_{r}^{n}))_{2}\le0.
\]
Using this, the Burkholder-Davis-Gundy inequality and \eqref{eq:sublinear_B_in_S} we obtain
\begin{align*}
\frac{1}{2}\E\sup_{t\in[0,T]}e^{-Ct}\|X_{t}^{n}\|_{2}^{2} & \le\E\|x_{0}\|_{2}^{2}-2\ve\E\int_{0}^{T}e^{-Cr}\|X_{r}^{n}\|_{H_{0}^{1}}^{2}dr+C.
\end{align*}
Passing to the (weak) limit yields the result.
\end{proof}

\section{Moreau-Yosida approximation of singular powers\label{sec:Moreau-Yosida}}

In this section we collect and prove some facts about the Moreau-Yosida approximation of certain monomials (cf. e.g. \cite[Section 2.2]{B10} for backgound on the Moreau-Yosida approximation). 

For $m\in[0,1]$ let $\psi(r):=\frac{1}{m+1}|r|^{m+1}$, $r\in\R$ and let $\psi^{\ve}:\R\to\R$ be the Moreau-Yosida approximation of $\psi$, i.e.
\[
\psi^{\ve}(r):=\inf_{s\in\R}\left\{ \frac{|r-s|^{2}}{2\ve}+\psi(s)\right\} .
\]
Then 
\[
\partial\psi^{\ve}=\phi^{\ve}(r):=\frac{1}{\ve}(r-J^{\ve}r)\in\phi(J^{\ve}r)\quad\forall r\in\R,
\]
is the Yosida approximation of $\phi=\partial\psi$, where $J^{\ve}=(I+\ve\phi)^{-1}$ is the resolvent of $\phi$ with $I$ denoting the identity map on $\R$. We note that
\begin{equation}
|\phi^{\ve}(r)|\le|\phi(r)|:=\inf\{|\eta|:\eta\in\phi(r)\}\quad\forall r\in\R.\label{eq:phi-delta-bound}
\end{equation}
Moreover,
\begin{align}
\psi^{\ve}(r) & =\frac{1}{2\ve}|r-J^{\ve}r|^{2}+\psi(J^{\ve}r)\label{eq:moreau_yos}\\
 & =\frac{\ve}{2}|\phi^{\ve}(r)|^{2}+\psi(J^{\ve}r)\quad\forall r\in\R\nonumber 
\end{align}
and thus
\begin{equation}
\psi(J^{\ve}r)\le\psi^{\ve}(r)\le\psi(r)\quad\forall r\in\R.\label{eq:MY-ineq}
\end{equation}
By the subgradient inequality we have
\[
\eta(J^{\ve}r-r)+\psi(r)\le\psi(J^{\ve}r)
\]
for all $\eta\in\phi(r)$. Hence, using the definition of $\phi^{\ve}$ 
\begin{align*}
\psi(r)-\psi(J^{\ve}r) & \le-\eta(J^{\ve}r-r)\\
 & \le|\eta|\ve|\phi^{\ve}(r)|
\end{align*}
for every $\eta\in\phi(r)$. Hence, using \eqref{eq:phi-delta-bound} and \eqref{eq:MY-ineq} we obtain
\begin{align}
|\psi(r)-\psi^{\ve}(r)| & \le\ve|\phi(r)|^{2}\label{eq:yosida_convergence-2}\\
 & \le C\ve(1+\psi(r))\quad\forall r\in\R.\nonumber 
\end{align}
We note that for all $a,b\in\R$
\begin{align*}
(\phi^{\ve_{1}}(a)-\phi^{\ve_{2}}(b))\cdot(a-b)= & (\phi^{\ve_{1}}(a)-\phi^{\ve_{2}}(b))\cdot(J^{\ve_{1}}a-J^{\ve_{2}}b)\\
 & +(\phi^{\ve_{1}}(a)-\phi^{\ve_{2}}(b))\cdot(a-J^{\ve_{1}}a-(b-J^{\ve_{2}}b))\\
\ge & (\phi^{\ve_{1}}(a)-\phi^{\ve_{2}}(b))\cdot(\ve_{1}\phi^{\ve_{1}}(a)-\ve_{2}\phi^{\ve_{2}}(b))\\
\ge & -\frac{1}{2}(\ve_{1}+\ve_{2})\left(|\phi^{\ve_{1}}(a)|^{2}+|\phi^{\ve_{2}}(b)|^{2}\right).
\end{align*}
Since
\[
|\phi^{\ve_{1}}(a)|^{2}\le|\phi(a)|^{2}\le C(1+|a|^{2})
\]
we conclude
\begin{align}
(\phi^{\ve_{1}}(a)-\phi^{\ve_{2}}(b))\cdot(a-b) & \ge-C(\ve_{1}+\ve_{2})(1+|a|^{2}+|b|^{2}).\label{eq:monotone_Y_bound}
\end{align}

\section{Relaxation of $L^{m+1}$ norms on $H^{-1}$\label{sec:heaviside_relaxation}}

For $m\ge0$ we define
\[
L^{m+1}\cap H^{-1}:=\left\{ v\in L^{m+1}|\int vhdx\le C\|h\|_{H_{0}^{1}},\ \forall h\in C_{c}^{1}(\mcO)\ \text{for some }C\ge0\right\} .
\]
By continuity, every $v\in L^{m+1}\cap H^{-1}$ the map $h\mapsto\int vhdx$ can be uniquely extended from $C_{c}^{1}$ to a bounded linear functional on $H_{0}^{1}$. Hence, $L^{m+1}\cap H^{-1}\subseteq H^{-1}$. We set 
\[
\vp(v):=\begin{cases}
\frac{1}{m+1}\|v\|_{m+1}^{m+1} & ,\quad v\in L^{m+1}\cap H^{-1}\\
+\infty & ,\quad H^{-1}\setminus(L^{m+1}\cap H^{-1}),
\end{cases}
\]

\begin{lem}
\label{lem:lsc_hull_lm}Assume $m>0$. Then $\vp$ is lower-semicontinuous on $H^{-1}$. \end{lem}
\begin{proof}
Let $v^{n}\in L^{m+1}\cap H^{-1}$ with $\|v^{n}\|_{m+1}^{m+1}\le C$ and $v^{n}\to v$ in $H^{-1}$. Then $v^{n}\rightharpoonup v$ in $L^{m+1}$ for some subsequence again denoted by $v^{n}$. Now $\|\cdot\|_{m+1}^{m+1}$ is weakly lower-semicontinuous on $L^{m+1}$ and thus
\[
\|v\|_{m+1}^{m+1}\le C.
\]

\end{proof}
Due to the lack of reflexivity of $L^{1}$ this argument fails in the case $m=0$. We next provide a characterization of the corresponding lower-semicontinuous hull of $\vp_{0}(\cdot):=\|\cdot\|_{1}$ on $H^{-1}$. 

Let $\mcM=\mcM(\mcO)$ be the space of signed Borel measures with finite total variation on $\mcO\subseteq\R^{d}$ and 
\[
\mcM\cap H^{-1}:=\left\{ \mu\in\mcM|\int_{\mcO}h(x)d\mu(x)\le C\|h\|_{H_{0}^{1}},\ \forall h\in C_{c}^{1}(\mcO)\ \text{for some }C\ge0\right\} .
\]
By continuity, for every $v\in\mcM\cap H^{-1}$ the map $h\mapsto\int hdv$ can be extended from $C_{c}^{1}$ to a (uniquely determined) bounded linear functional on $H_{0}^{1}$. The resulting map $\iota:\mcM\cap H^{-1}\to H^{-1}$ thus is injective. Hence, $\mcM\cap H^{-1}\subseteq H^{-1}$ and in the following we identify $\mcM\cap H^{-1}$ with its embedding into $H^{-1}$ , except for the proof of Lemma \ref{lem:lsc_hull_l1} below, where for $v\in\mcM\cap H^{-1}$ with $v=\iota(\mu)$ we write $v_{\mu}:=\iota(\mu)$. We extend $\mu\in\mcM(\mcO)$ to all of $\R^{d}$ by zero.
\begin{lem}
\label{lem:lsc_hull_l1}Let $m=0$ and $\vp$ be the lower-semicontinuous hull of $\vp_{0}(\cdot):=\|\cdot\|_{1}$ on $H^{-1}$. Then
\[
\vp(v):=\begin{cases}
\TV(\mu) & ,\quad v=v_{\mu}\in\mcM\cap H^{-1}\\
+\infty & ,\quad otherwise.
\end{cases}
\]
\end{lem}
\begin{proof}
We first prove that $\vp$ is weakly (hence strongly) lower-semicontinuous on $H^{-1}$. We recall that $\mcM$ is the dual $C_{0}(\mcO)^{*}$ of $C_{0}(\mcO)$, i.e. the set of all continuous functions vanishing onf the boundary $\partial\mcO$, equipped with the sup-norm. Let $v_{\mu^{n}}\in\mcM\cap H^{-1}$ with $\vp(v_{\mu^{n}})\le C$ and $v_{\mu^{n}}\rightharpoonup v$ in $H^{-1}$. Since $C_{0}(\mcO)$ is separable and $\mu^{n}$ has uniformly bounded total variation we have
\[
\mu^{n}\rightharpoonup\td\mu
\]
weakly$^{*}$ in $\mcM$ for some subsequence of $\mu^{n}$ and some $\td\mu\in\mcM$. Taking the limit $n\to\infty$ in
\[
v_{\mu^{n}}(h)=\int_{\mcO}h(x)d\mu^{n}(x),\quad h\in C_{c}^{1}
\]
and using $v_{\mu^{n}}\rightharpoonup v$ in $H^{-1}$ yields
\[
v(h)=\int_{\mcO}h(x)d\td\mu(x)\quad\forall h\in C_{c}^{1}
\]
and thus $v=v_{\td\mu}\in\mcM\cap H^{-1}$. Since the total variation norm is lower semicontinuous with respect to weak convergence we obtain
\[
\vp(v)=\TV(\td\mu)\le\liminf_{n\to\infty}\TV(\mu^{n})=\liminf_{n\to\infty}\vp(v_{\mu^{n}})\le C,
\]
which proves $\vp$ to be weakly lower-semicontinuous on $H^{-1}$. 

It remains to prove that $\vp$ is the lower-semicontinuous hull on $H^{-1}$ of $\vp$ restricted to $L^{1}\cap H^{-1}$. Let $v\in H^{-1}$ and let $v^{\ve}:=J^{\ve}v\in L^{2}$, where $J^{\ve}:=(I-\ve\D)^{-1}$ is the resolvent of $-\D$ in $H^{-1}$. Note that $v^{\ve}\to v$ in $H^{-1}$, where we identify $v^{\ve}\in L^{2}$ with its embedding into $H^{-1}$ via
\[
v^{\ve}(h)=\int_{\mcO}v^{\ve}hdx,\quad h\in H_{0}^{1}.
\]
 For $v\in L^{2}\subseteq H^{-1}$ we observe
\begin{align*}
(J^{\ve}v)(h) & =\int_{\mcO}(J^{\ve}v)hdx\\
 & =\int_{\mcO}vJ^{\ve}hdx\\
 & =v(J^{\ve}h)\quad\forall h\in H_{0}^{1}.
\end{align*}
By the density of $L^{2}$ in $H^{-1}$ this yields
\[
J^{\ve}v=v\circ J^{\ve}\quad\forall v\in H^{-1}.
\]
Moreover, we recall $J^{\ve}:C_{0}(\mcO)\to C_{0}(\mcO)$, $\|J^{\ve}h\|_{C_{0}}\le\|h\|_{C_{0}}$ for all $h\in C_{0}$ and for $h\in C_{c}^{\infty}(\mcO)$ we have
\begin{equation}
J^{\ve}h\to h\quad\text{in }C_{0}.\label{eq:res_conv_in_sup}
\end{equation}
Let $\mu\in\mcM=(C_{0}(\mcO))^{*}$. We define $\mu^{\ve}(h)=(J^{\ve}\mu)(h):=\mu(J^{\ve}h)$ for $h\in C_{0}$ and identify $\mu^{\ve}$ with its representation in $\mcM$. Hence,
\begin{align*}
\int_{\mcO}hd\mu^{\ve} & =\int_{\mcO}J^{\ve}hd\mu
\end{align*}
for all $h\in C_{0}$ and thus
\begin{align*}
TV(\mu^{\ve}) & =\sup_{\|h\|_{C_{0}\le1}}\mu^{\ve}(h)\\
 & \le TV(\mu)\sup_{\|h\|_{C_{0}\le1}}\|J^{\ve}h\|_{C_{0}}\\
 & \le TV(\mu).
\end{align*}
Consequently, there is a subsequence $\mu^{\ve_{n}}\rightharpoonup\td\mu$ weakly$^{*}$ in $\mcM$ and due to \eqref{eq:res_conv_in_sup} this means $\mu^{\ve_{n}}\rightharpoonup\mu$ weakly$^{*}$ in $\mcM$ and
\begin{align*}
TV(\mu) & \le\liminf_{n\to\infty}TV(\mu^{\ve_{n}})\le TV(\mu).
\end{align*}
A standard contradiction argument then yields
\begin{equation}
\lim_{\ve\to0}TV(\mu^{\ve})=TV(\mu).\label{eq:TV_convergence}
\end{equation}
Let now $v_{\mu}\in\mcM\cap H^{-1}$. Then
\begin{align*}
(J^{\ve}v_{\mu})(h) & =v_{\mu}(J^{\ve}h)\\
 & =\int_{\mcO}J^{\ve}hd\mu\\
 & =\int_{\mcO}hd\mu^{\ve},
\end{align*}
for all $h\in C_{c}^{1}$ and thus $J^{\ve}v_{\mu}\in\mcM\cap H^{-1}$ with $J^{\ve}v_{\mu}=v_{J^{\ve}\mu}$. Due to \eqref{eq:TV_convergence} this implies
\begin{align*}
\vp(J^{\ve}v_{\mu}) & =\vp(v_{J^{\ve}\mu})\\
 & =TV(J^{\ve}\mu)\\
 & \to TV(\mu)\\
 & =\vp(v_{\mu}).
\end{align*}
Since $J^{\ve}v_{\mu}\in L^{2}\subseteq L^{1}\cap H^{-1}$ and $J^{\ve}v_{\mu}\to v_{\mu}$ in $H^{-1}$, $\vp$ is the lower-semicontinuous hull on $H^{-1}$ of $\vp$ restricted to $L^{1}\cap H^{-1}$. \end{proof}
\begin{lem}
\label{lem:subgradient}Let $m\in[0,1]$ and $v\in L^{m+1}\cap H^{-1}$. Then
\[
\partial\vp(v)\supseteq\{-\D w:\ w\in H_{0}^{1},\ w\in\phi(v)\text{ a.e.}\}.
\]
\end{lem}
\begin{proof}
\textbf{Case $m>0$:} Let $J_{\ve}$, $\ve>0$ be as in the proof of Lemma \ref{lem:lsc_hull_l1}. Assume that $w=\phi(v)=|v|^{m-1}v\in H_{0}^{1}$. Then $w\in L^{\frac{m+1}{m}}$ and we have to show that 
\[
\vp(v)\le(-\D w,v-y)+\vp(y)\quad\forall y\in L^{m+1}\cap H^{-1}.
\]
As in the proof of Lemma \ref{lem:lsc_hull_l1} we have for all $y\in L^{m+1}\cap H^{-1}$
\begin{align*}
 & \vp(v)-(-\D w,v-y)_{H^{-1}}\\
 & =\vp(v)-\lim_{\ve\to0}(-\D w,J_{\ve}(v-y))_{H^{-1}}\\
 & =\vp(v)-\lim_{\ve\to0}(w,J_{\ve}(v-y))_{2}\\
 & =\vp(v)-\int_{\mcO}w(v-y)dx,
\end{align*}
where we used that $J_{\ve}v\to v$ in $L^{p}$ as $\ve\to0$ for every $v\in L^{p}$ and all $p\in[1,\infty)$. The last expression equals
\begin{align*}
 & \frac{1}{m+1}\int_{\mcO}|v|^{m+1}dx-\int_{\mcO}w(v-y)dx\\
 & \le\frac{1}{m+1}\int_{\mcO}|v|^{m+1}dx-\int_{\mcO}|v|^{m+1}dx+\int_{\mcO}|v|^{m-1}vydx\\
 & \le\vp(y),
\end{align*}
where we used Hölder's and Young's inequality in the last step. This finishes the proof.

\textbf{Case $m=0$:} Assume $w\in\phi(v)$ a.e. with $w\in H_{0}^{1}$. Arguing as in the case $m>0$ and using $|w|\le1$ a.e. we have
\begin{align*}
\vp(v)-(-\D w,v-y)_{H^{-1}} & =\vp(v)-\int_{\mcO}w(v-y)dx\\
 & \le\int_{\mcO}|y|dx,
\end{align*}
which finishes the proof. $ $
\end{proof}
\bibliographystyle{plain}
\bibliography{/home/benni/cloud/current_work/latex-refs/refs}

\end{document}